\documentclass[12pt,english]{amsart}
\usepackage{lmodern}
\usepackage{helvet}

\usepackage[T1]{fontenc}
\usepackage[latin9]{inputenc}
\usepackage{mathrsfs}
\usepackage{amsthm}
\usepackage{amstext}
\usepackage{amssymb}
\usepackage{graphicx}
\usepackage{esint}

\makeatletter

\newcommand{\noun}[1]{\textsc{#1}}

\numberwithin{equation}{section}
\numberwithin{figure}{section}
  \theoremstyle{remark}
  \newtheorem*{rem*}{\protect\remarkname}
\theoremstyle{plain}
\newtheorem{thm}{\protect\theoremname}[section]
  \theoremstyle{plain}
  \newtheorem{lem}[thm]{\protect\lemmaname}
  \theoremstyle{plain}
  \newtheorem{prop}[thm]{\protect\propositionname}
  \theoremstyle{definition}
  \newtheorem{example}[thm]{\protect\examplename}
  \theoremstyle{remark}
  \newtheorem{rem}[thm]{\protect\remarkname}
  \theoremstyle{definition}
  \newtheorem{defn}[thm]{\protect\definitionname}
  \theoremstyle{plain}
  \newtheorem{cor}[thm]{\protect\corollaryname}

\usepackage{amsfonts}
\usepackage{amssymb}
\usepackage[all]{xy}
\usepackage{array}
\usepackage{lmodern}
\usepackage[T1]{fontenc}

\def\F{\mathcal{F}}
\def\QQ{\mathbb{Q}}
\def\RR{\mathbb{R}}
\def\CC{\mathbb{C}}

\def\ZZ{\mathbb{Z}}
\def\PP{\mathbb{P}}

\def\del{\partial}

\def\k12{\mathcal{K}_{\lambda_1,\lambda_2}}
\def\tk12{\tilde{\mathcal{K}}_{\lambda_1,\lambda_2}}
\def\ck12{\check{\mathcal{K}}_{\lambda_1,\lambda_2}}
\def\Li{Li}
\def\si{\sigma}
\def\la{\lambda}
\def\x{\mathrm{X}}
\def\y{\mathrm{Y}}
\def\z{Z_{\Delta}}
\def\n{N_{\Delta}}

\theoremstyle{definition}

\theoremstyle{definition}

\theoremstyle{theorem}

\theoremstyle{theorem}

\theoremstyle{theorem}

\theoremstyle{definition}

\makeatother

\usepackage{babel}
  \providecommand{\corollaryname}{Corollary}
  \providecommand{\definitionname}{Definition}
  \providecommand{\examplename}{Example}
  \providecommand{\lemmaname}{Lemma}
  \providecommand{\propositionname}{Proposition}
  \providecommand{\remarkname}{Remark}
\providecommand{\theoremname}{Theorem}

\newcommand{\Res}{\operatorname{Res}}

\newcommand{\ord}{\operatorname{ord}}
\newcommand{\cD}{\mathcal{D}}
\newcommand{\cP}{\mathcal{P}}
\newcommand{\cS}{\mathcal{S}}
\newcommand{\cT}{\mathcal{T}}

 \usepackage[usenames,dvipsnames]{color}

\begin{document}

\title{Simplicial Abel-Jacobi maps and reciprocity laws}

\author[Burgos Gil, Kerr, Lewis, and Lopatto]{Matt Kerr, James Lewis and Patrick Lopatto\\
with an appendix by Jos\'e Ignacio Burgos-Gil }


\begin{abstract}
We describe an explicit morphism of complexes that induces the cycle-class
maps from (simplicially described) higher Chow groups to rational
Deligne cohomology. The reciprocity laws satisfied by the currents
we introduce for this purpose are shown to provide a clarifying perspective
on functional equations satisfied by complex-valued di- and trilogarithms.

\tableofcontents{}
\end{abstract}

\maketitle

\section{Introduction}

Abel-Jacobi maps for higher Chow groups\begin{equation}\label{eqn intro (*)}AJ_{\x}^{p,n}:\, CH^p(\x,n)_{\QQ} \to H^{2p-n}_{\mathcal{H}}\left( \x^{an}_{\CC},\QQ(p)\right)
\end{equation}were introduced (for smooth quasi-projective $\x$ over $k\subset\CC$)
in \cite{Ke1,KLM} via an extension of Griffiths's formula for $n=0$
to a quasi-isomorphic subcomplex of the cubical Bloch complex $Z^{p}(\x,\bullet)_{\QQ}$.
Together with their extension to motivic cohomology $H_{\mathcal{M}}^{2p-n}(\x,\QQ(n))$
in the singular case \cite{KL}, these $AJ$-maps have been used (for
example) to interpret limits of normal functions of geometric origin
\cite{GGK}, to study toric and Eisenstein symbols on families of
Calabi-Yau varieties \cite{DK}, to compute a family of Feynman integrals
\cite{BKV}, and to study torsion in $CH^{p}(\x,n)$ \cite{Pe} (though
an integral moving lemma is still missing for this to work in general).

\subsection*{Summary}

The main purpose of this paper is to give an alternate formula (Theorem
\ref{AJ Thm}) for \eqref{eqn intro (*)} on the \emph{simplicial}
Bloch complex $Z_{\Delta}^{p}(\x,\bullet)_{\QQ}$, sending a precycle%
\footnote{We use the term ``precycle'' for an element of Bloch's complex,
and ``{[}higher Chow{]} cycle'' for an element of $CH^{p}(\x,n)_{\QQ}$
(equivalence class of a closed precycle).%
} $\mathfrak{Z}$ to a triple of currents $(T_{\mathfrak{Z}}^{\Delta},\Omega_{\mathfrak{Z}}^{\Delta},R_{\mathfrak{Z}}^{\Delta})$
(cf. \eqref{eqn 3.3}) on $\x$. (We shall restrict for simplicity
to the smooth projective case, so that the absolute Hodge cohomology
in \eqref{eqn intro (*)} is just Deligne cohomology; the generalization
to quasi-projective is exactly as in $\S3$ of \cite{KL}.) This had
been a goal of the authors of \cite{KLM}, but seemed out of reach
at the time. The basic currents on $\mathbb{P}^{n}$ we develop for
this purpose in $\S$2 (see \eqref{eqnRn}, \eqref{eqnS}, and \eqref{eqnRsimp})
lead to two ``simplicial'' reciprocity laws (Theorems \ref{thm rec law A}
and \ref{thm Rec Law B}) for their integrals over subvarieties of
projective space, which are applied to directly recover functional
equations for \emph{complex}-valued di- and tri-logarithms in $\S\S$5-6.
The second of these laws takes a very intriguing form, and leads to
a more straightforward proof of the Kummer-Spence relation \eqref{6eq1}
than for the real-valued trilogarithm in \cite{Go3}. This paper is
written in such a way that the reader interested only in these applications
can skip $\S3$ entirely.

The $AJ$ formulas of \cite{KLM} were based on the \emph{cubical}
higher Chow complex for several reasons, including the greater ease
of constructing good currents on $\square^{n}$ (not to mention explicit
cycles in the cubical complex), the availability of bounding membranes
(to provide a link to the extension class definition of $AJ$), and
the greater naturality of cup-products in the cubical setting. On
the other hand, the simplicial formulation of higher Chow groups allows
for linear higher cycles, which provide direct links to the seminal
work of Goncharov on polylogarithms (cf. \cite{Go1}-\cite{Go4})
and to the cohomology of the general linear group \cite{dJ}. These
special features make a compelling argument for revisiting \cite{KLM}
from the simplicial point of view.

Moreover, we point out that, up to this point, there has not even
been a correct real regulator formula on the simplicial level. While
it was checked in $\S$3.1 of \cite{Ke1} that Goncharov's currents
in \cite{Go1} yield the real regulator (i.e., composition of $AJ_{\x}^{p,n}$
with $\pi_{\RR}:\, H_{\mathcal{H}}^{2p-n}(\x,\QQ(p))\to H_{\mathcal{H}}^{2p-n}(\x,\RR(p))$)
on the cubical complex, it turns out that the simplicial version constructed
in $\S$6.1 of \cite{Go1} (or $\S$2 of \cite{Go2}) is neither well-defined
nor a map of complexes. The main problem is that $Z_{\Delta}^{p}(\x,n)_{\QQ}$
is a subgroup of $Z^{p}(\x\times\PP^{n})_{\QQ}/Z^{p}(\x\times\mathrm{H}_{n})_{\QQ}$,
where $\mathrm{H}_{n}\subset\PP^{n}$ is the hyperplane defined by
$X_{0}+\cdots+X_{n}=0$,%
\footnote{The Roman script $\mathrm{X}$ and $\mathrm{Y}$ are used throughout
to denote varieties, while $X,Y$ denote projective coordinates; this
convention is not followed for other letters.%
} and the currents of {[}op. cit.{]} do not vanish on $Z^{p}(\x\times\mathrm{H}_{n})_{\QQ}$.
For more details, see Remark \ref{rem Go}.

\subsection*{Some motivation}

On any $\PP^{m}$, with coordinates $[Y_{0}:\cdots:Y_{m}]$ (and $y_{i}:=Y_{i}/Y_{0}$),
one has a natural $\text{dlog}$-form
\[
\Omega_{m}^{\Delta}(Y_{0}:\ldots:Y_{m}):=\Omega_{m}\left(\tfrac{Y_{1}}{Y_{0}},\ldots,\tfrac{Y_{m}}{Y_{0}}\right):=\tfrac{dy_{1}}{y_{1}}\wedge\cdots\wedge\tfrac{dy_{m}}{y_{m}}
\]
and real $m$-chain (see \eqref{eqn T} for orientation)
\[
T_{m}^{\Delta}(Y_{0}:\cdots:Y_{m}):=\{y_{1},\ldots,y_{m}\in\RR_{+}\};
\]
we may regard both as $m$-currents. The constructions of this paper
center around the existence of sequences of $(m-1)$-currents
\[
\mathfrak{R}_{m}\in D^{m-1}(\PP^{m})\;\;\;\;\;(m\ge1)
\]
satisfying two properties. To motivate the first property (\eqref{eqn !1p4}
below), consider the family
\[
\mathrm{X}_{s}:=\left\{ \prod_{i=0}^{2n}X_{i}-s\sum_{i=0}^{2n}X_{i}^{2n+1}=0\right\} \subset\PP^{2n}
\]
of Calabi-Yau $(2n-1)$-folds,%
\footnote{We shall ignore the fact that this family is not semistable at $s=0$.%
} and let $\PP^{n}\cong\mathrm{P}_{t}\subset\PP^{2n}$ be a family
of linear $n$-planes%
\footnote{Think of $t$ as varying in some neighborhood of $0$ in a $\CC^{M}$,
with $\Gamma_{t}$ the union of $\{\mathrm{P}_{t}\}$ over a radial
segment $\overrightarrow{0.t}$.%
} with $\partial\Gamma_{t}=\mathrm{P}_{t}-\mathrm{P_{0}}$ ($\Gamma_{t}=(2n+1)$-chain).
Setting $\mathfrak{Z}_{s,t}:=\mathrm{P}_{t}\cdot\mathrm{X}_{s}$,
we have $0=[\mathfrak{Z}_{s,t}-\mathfrak{Z}_{s,0}]\in CH^{n}(\mathrm{X}_{s})$;
in particular, writing $\omega_{s}:=Res_{\mathrm{X}_{s}}\left(F_{s}^{-1}\sum_{i=0}^{2n}(-1)^{i}dX_{0}\wedge\cdots\wedge\widehat{dX_{i}}\wedge\cdots\wedge dX_{2n}\right),$
the Griffiths Abel-Jacobi integrals $\int_{\Gamma_{t}\cdot\mathrm{X_{s}}}\omega_{s}$
\emph{vanish}. Taking $s\to0$ and writing $\mathrm{X}_{0}=\cup\mathrm{X}_{0}^{j}$,
where $\mathrm{X}_{0}^{j}=\{X_{j}=0\}\cong\PP^{2n-1}\underset{\rho_{j}}{\hookrightarrow}\PP^{2n}$,
we obtain \begin{equation}\label{eqn ??p3}0=\int_{\Gamma_t \cdot \mathrm{X}_0} \omega_0 = \sum_{j=0}^{2n}(-1)^j\int_{\Gamma_t\cdot \mathrm{X}_0^j }\Omega^{\Delta}_{2n-1}\left( X_0:\cdots:\widehat{X_j}:\cdots:X_{2n}\right) .
\end{equation}

Now suppose that for each $m$ \begin{equation}\label{eqn !1p4}d[\mathfrak{R}_m] = \Omega^{\Delta}_m - (2\pi i)^m\delta_{T^{\Delta}_m} + 2\pi i\sum_{j=0}^m (-1)^j {\rho_j}_* \mathfrak{R}_{m-1}
\end{equation}holds. Then on $\mathrm{X}_{0}$, $\omega_{0}\equiv d[\sum(-1)^{j}{\rho_{j}}_{*}\mathfrak{R}_{2n-1}]$
modulo $\ZZ(2n-1)$-valued currents, and by Stokes's theorem \eqref{eqn ??p3}
gives
\[
0=\sum_{j=0}^{2n}(-1)^{j}\int_{\mathfrak{Z}_{0,t}^{j}-\mathfrak{Z}_{0,0}^{j}}\mathfrak{R}_{2n-1}
\]
(where $\mathfrak{Z}_{0,t}^{j}:=\mathfrak{Z}_{0,t}\cdot\mathrm{X}_{0}^{j}$),
which shows that \begin{equation}\label{eqn ?p4}\sum_{j=0}^{2n}(-1)^j\int_{\mathfrak{Z}_{0,t}^j} \mathfrak{R}_{2n-1} \; \; \;\; \text{is constant.}
\end{equation}In fact,  according to the first of the ``reciprocity laws'' in
$\S4$, this constant belongs to $\ZZ(2n-1)$, and the proof is simpler
than the argument just given. (The second of the two laws, however,
is more subtle.) If $n=1$ and the $\mathrm{P}_{t}$ are lines, \eqref{eqn ?p4}
is just $\log(x)-\log(y)+\log(y/x)\underset{\ZZ(1)}{\equiv}0$.

To motivate the second property, suppose we would like to have lifts
$\tilde{\varepsilon}_{n}\in H_{meas}^{2n-1}\left(GL_{n}(\CC),\CC/\ZZ(n)\right)$
of the Borel classes 
\[
\varepsilon_{n}\in H_{cont}^{2n-1}\left(GL_{n}(\CC),\RR\right):=H^{2n-1}\left\{ \text{Cont}\left(GL_{n}(\CC)^{\times(\bullet+1)}\right)^{GL_{n}(\CC)},\delta\right\} 
\]
(e.g., $\varepsilon_{1}(g_{0},g_{1})\propto\log|g_{1}/g_{0}|$ and
\[
\varepsilon_{2}(g_{0},g_{1},g_{2},g_{3})\propto D(CR([g_{0}v],[g_{1}v],[g_{2}v],[g_{3}v])),
\]
where $D$ is the Bloch-Wigner function and $v\in\CC^{2}$ is fixed).
For instance, one might use such lifts to detect elements (particularly
torsion ones) of $H_{2n-1}(GL_{n}(\mathbb{F}),\ZZ)$ or to construct
complex lifts of hyperbolic volume. Recall that Bloch's higher Chow
complexes were originally defined in their simplicial formulation:
writing $\Delta^{n}:=\PP^{n}\backslash\mathrm{H}_{n}$, $\partial\Delta^{n}:=\cup_{j=0}^{n}\rho_{j}\left(\PP^{n-1}\backslash\mathrm{H}_{n-1}\right)$,
$(\partial)\Delta_{\mathrm{X}}^{n}:=\mathrm{X}\times(\partial)\Delta^{n}$,
the subgroups $Z_{\Delta}^{p}(\mathrm{X},n)\leq Z^{p}(\Delta_{\mathrm{X}}^{n})=\tfrac{Z^{p}(\mathrm{X}\times\PP^{n})}{Z^{p}(\mathrm{X}\times\mathrm{H}_{n})}$
(generated by subvarieties meeting faces of $\partial\Delta_{\mathrm{X}}^{n}$
properly) form a complex $Z_{\Delta}^{p}(\mathrm{X},\bullet)_{\QQ}$
under $\partial=\sum(-1)^{j}\rho_{j}^{*}$ with homology $CH^{p}(\mathrm{X},n)_{\QQ}$.
The relevant case is where $\mathrm{X}=Spec(\CC)$ is a point.

If the $\{\mathfrak{R}_{m}\}$ satisfy the additional property\begin{equation} \label{p5!2}
\mathfrak{R}_m |_{\mathrm{H}_m} =0,
\end{equation}then we can use them to induce Abel-Jacobi maps\begin{equation} \label{p5!!!}
CH^n (Spec(\CC),2n-1) \overset{AJ^{\Delta}}{\longrightarrow} \CC/\ZZ(n) 
\end{equation}by integrating $(2\pi i)^{n-1}\mathfrak{R}_{2n-1}$ over a cycle $\mathfrak{Z}$.
Composing this with the map
\[
H_{2n-1}\left(GL_{n}(\CC),\ZZ\right)\longrightarrow CH^{n}\left(Spec(\CC),2n-1\right)
\]
defined by%
\footnote{See \cite{dJ} for details of this construction.%
} fixing $v\in\CC^{n}$ and sending a tuple $\underline{g}:=(g_{0},\ldots,g_{2n-1})\in GL_{n}(\CC)^{\times2n}$
(in general position) to
\[
\mathfrak{Z}_{\underline{g}}:=\left\{ \left(\begin{array}{ccc}
\uparrow &  & \uparrow\\
g_{0}v & \cdots & g_{2n-1}v\\
\downarrow &  & \downarrow
\end{array}\right)\left(\begin{array}{c}
X_{0}\\
\vdots\\
X_{2n-1}
\end{array}\right)=0\right\} \subset\Delta^{2n-1},
\]
we apparently obtain a candidate for $\tilde{\varepsilon}_{n}$. In
fact, this is a bit glib as $AJ^{\Delta}$ is not defined on all of
$Z_{\Delta}^{2n-1}(Spec(\CC),\bullet)$, but only on a subcomplex
$Z_{\Delta,\RR}^{2n-1}(Spec(\CC),\bullet)$ of precycles well-behaved
with respect to the currents. So far we only know this subcomplex
is \emph{rationally} quasi-isomorphic (Proposition \ref{ML1}), and
so \eqref{p5!!!} only maps to $\CC/\QQ(n)$. Nevertheless, the direct
formula 
\[
\tilde{\varepsilon}_{n}(\underline{g}):=(2\pi i)^{1-n}\int_{\mathfrak{Z}_{\underline{g}}}\mathfrak{R}_{2n-1}
\]
appears to give a measurable cohomology class with $\CC/\ZZ(n)$ coefficients,
which should be investigated further.

Finally, to give the reader a flavor of what sort of concrete computation
is possible with our $AJ$ formula in the simplest case, where $\mathrm{X}$
is a point over a number field, consider the element $\mathfrak{Z}\in Z_{\Delta}^{2}\left(Spec(\QQ(\zeta)),3\right)_{\QQ}$
($\zeta=e^{\frac{2\pi i}{3}}$) defined by $\mathfrak{Z}:=\mathfrak{Z}_{1}+\mathfrak{Z}_{2}:=$
\[
\left[-Z(Z-\zeta^{2}W)^{2}:W(Z-\zeta W)(Z-\zeta^{2}W):-W^{3}:Z^{3}\right]_{[Z:W]\in\PP^{1}}
\]
\[
+\left[-3Z(Z+\zeta^{2}W):3Z^{2}:-W^{2}:W^{2}\right]_{[Z:W]\in\PP^{1}}.
\]
We have $\partial\mathfrak{Z}=-[3\zeta:-1:1]+[3\zeta:-1:1]=0$,%
\footnote{See \eqref{boundary}. Note that the intersections with coordinate
hyperplanes which lie inside $\mathrm{H}_{3}$ do not count, since
$\rho_{j}^{*}\mathrm{H}_{3}=\mathrm{H}_{2}$ and $Z_{\Delta}^{2}(\mathrm{X},2)\leq Z^{2}(\mathrm{X}\times\PP^{2})/Z^{2}(\mathrm{X}\times\mathrm{H}_{2}).$%
} and so this defines a higher Chow cycle $[\mathfrak{Z}]$, whose
image under the (simplicially defined) map
\[
AJ_{\mathrm{X}}^{2,3}:\, CH^{2}\left(Spec(\QQ(\zeta)),3\right)\to\CC/\QQ(2)
\]
is computed by integrating the 2-current 
\[
\mathfrak{R}:=\tfrac{1}{2\pi i}R_{3}\left(\tfrac{X_{1}+X_{2}+X_{3}}{-X_{0}},\tfrac{X_{2}+X_{3}}{-X_{1}},\tfrac{X_{3}}{-X_{2}}\right)
\]
on $\PP^{3}$ (cf. \eqref{eqnRn}) over $\mathfrak{Z}$. Since $-\tfrac{X_{3}}{X_{2}}\equiv1$
on $\mathfrak{Z}_{2}$, we have\begin{align*}
AJ([\mathfrak{Z}]) &= \int_{\mathfrak{Z}_1} \mathfrak{R} \\ &= \tfrac{1}{2\pi i} \int_{\PP^1} R_3\left( \tfrac{Z^3 - W^3 + W(Z-\zeta W)(Z-\zeta^2 W)}{Z(Z-\zeta^2 W)^2},\tfrac{W^3 - Z^3}{W(Z-\zeta W)(Z-\zeta^2 W)} ,\tfrac{Z^3}{W^3} \right) \\ &=  \tfrac{1}{2\pi i} \int_{\PP^1} R\left(\tfrac{t-\zeta}{t-\zeta^2},1-t,t^3\right) \\ &=  3\int_{\zeta}^{\zeta^2} \log(1-t)\text{dlog}(t) \\ &=  3\left( Li_2(\zeta^2)-Li_2(\zeta) \right) \\ &=  -3\sqrt{3} L\left( \chi_{-3},2\right) .
\end{align*}
\begin{rem*}
Throughout this paper, all cycle groups are taken with rational coefficients;
henceforth, we drop the subscript $\QQ$ used above. This choice reflects
the fact that we do not yet know how to prove Propositions \ref{ML1}
and \ref{ML2} (or some substitute) integrally. (This will be necessary
to enjoy the real benefits of the simplicial $AJ$ map when $\mathrm{X}$
is the spectrum of a number field, since in this case the main point
of lifting from $\RR(n-1)$ to $\CC/\ZZ(n)$ is probably to extract
torsion information.) Also note that in sections 3 and 6 we have relegated
to appendices those technical details which we judged to interrupt
the main line of argument (proofs of moving lemmas, etc.)
\end{rem*}

\subsubsection*{Acknowledgments:}

The authors wish to thank H. Gangl, B. Kahn, S. M\"uller-Stach, and especially J. Burgos
Gil for discussions related to this work; and to emphasize their debt
to A. Goncharov, whose work has inspired much of what we have tried
to do. They gratefully acknowledge partial support through NSF grants
DMS-1068974 and DMS-1361147 (M.K.), a grant from the Natural Sciences
and Engineering Research Council of Canada (J.L.), and from Washington
University's Office of Undergraduate Research (P.L.). Much of this
paper was written while M.K. was a member of the Institute for Advanced
Study, and he thanks the IAS for excellent working conditions and
the Fund for Mathematics for financial support.

\section{Two classes of simplicial currents}
\label{sec:two-class-simpl}

The explicit formulas for Abel-Jacobi maps for higher Chow groups
in \cite{KLM} were enabled by the construction of triples $(R_{n},\Omega_{n},T_{n})$
of currents on each $(\mathbb{P}^{1})^{n}$ with the telescoping property
\begin{equation}\label{eqnTelescope}d[R_n] = \Omega_n -(2\pi i)^n\delta_{T_n}+2\pi i \sum_{j=1}^n (-1)^{j}\left(  (\imath^0_j )_* R_{n-1} - (\imath^{\infty}_j)_* R_{n-1} \right), 
\end{equation}where $\imath_{j}^{\epsilon}:(\PP^{1})^{n-1}\hookrightarrow(\PP^{1})^{n}$
are the inclusions of the coordinate hyperplanes $z_{j}=\epsilon$.
We briefly recall their definition: let $T_{f}:=f^{-1}(\mathbb{R}_{<0})$
be oriented so that $\partial T_{f}=(f)$, and $\log(\cdot)$ denote
the discontinuous function with $\arg\in(-\pi,\pi]$. Writing $\varepsilon:=(-1)^{n-1}2\pi i$,
we set \begin{equation}\label{eqnRn}\begin{matrix} R_n := R_n (z_1, \ldots ,z_n) := \\ \\ \sum_{j=1}^n \varepsilon^{j-1} \log(z_j) \frac{dz_{j+1}}{z_{j+1}} \wedge \cdots \wedge \frac{dz_n}{z_n} \cdot \delta_{ T_{z_1}\cap \cdots \cap T_{z_{j-1}} }\in \mathcal {D}^{n-1}\left( (\PP^1)^n\right),  \end{matrix} \end{equation} 
\begin{equation} \Omega_n :=\Omega_n(z_1,\ldots ,z_n):=\frac{dz_1}{z_1}\wedge \cdots \wedge \frac{dz_n}{z_n}\in \mathcal D^{n,0}\left( (\PP^1)^n\right), \end{equation} and
\begin{equation}\label{eqn T}T_n:=T_n(z_1,\ldots ,z_n):=T_{z_1}\cap\cdots\cap T_{z_n}\in C_{\text{top}}^n\left((\PP^1)^n \right). \end{equation}Roughly
speaking, \eqref{eqnTelescope} follows from $d[\log z_{j}]=\frac{dz_{j}}{z_{j}}-2\pi i\delta_{T_{z_{j}}}$
and $d\left[\frac{dz_{j}}{2\pi iz_{j}}\right]=\delta_{(z_{j})}=d\left[\delta_{T_{z_{j}}}\right]$.

The key point is \eqref{eqnRn}, which was arrived at in \cite{Ke1,Ke2}
by formally applying P. Griffiths's formula for $AJ$ \cite{Gr} to
relative cycles on the Cartesian product of a smooth projective $d$-fold
$\x$ with 
\[
(\square^{n},\partial\square^{n}):=\left((\PP^{1}\backslash\{1\})^{n},\cup_{j,\epsilon}\imath_{j}^{\epsilon}\left((\PP^{1}\backslash\{1\})^{n-1}\right)\right).
\]
 Writing $p_{j}(z_{1},\ldots,z_{n}):=(z_{1},\ldots,\widehat{z_{j}},\ldots,z_{n})\in\square^{n-1}$,
the cubical higher Chow precycles 
\[
Z^{p}(\x,n)\subset Z^{p}(\x\times\square^{n})/\overset{\text{degenerate cycles}}{\overbrace{\sum_{j}p_{j}^{*}Z^{p}(\x\times\square^{n-1})}}
\]
are the algebraic cycles meeting arbitrary intersections of the $\x\times\imath_{j}^{\epsilon}(\square^{n-1})$
properly, i.e. in the expected dimension (or less). The good precycles
$Z_{\RR}^{p}(\x,n)\subset Z^{p}(\x,n)$ are those which meet $T_{z_{1}}$,
$T_{z_{1}}\cap T_{z_{2}}$, $\ldots$, $T_{z_{1}}\cap\cdots\cap T_{z_{n}}$
and their arbitrary intersections with the $\x\times\imath_{j}^{\epsilon}(\square^{n-1})$
properly as well \cite{KL}. Given $\mathfrak{Z}\in Z_{\RR}^{p}(\x,n)$,
the convergence of
\[
\int_{\x}R_{\mathfrak{Z}}\wedge\omega:=\int_{\widetilde{\mathfrak{Z}}}R_{n}(\underline{z})\wedge\pi_{\x}^{*}\omega
\]
for arbitrary $\omega\in A^{2d-2p+n+1}(\x)$ defines a current $R_{\mathfrak{Z}}\in\mathcal{D}^{2p-n-1}(\x)$;
indeed, this holds on the level of summands of \eqref{eqnRn}. Similarly,
one defines $\Omega_{\mathfrak{Z}}\in F^{p}\mathcal{D}^{2p-n}(\x)$
and 
\[
T_{\mathfrak{Z}}:=(\pi_{\x})_{*}\left\{ (\x\times T_{n})\cap\mathfrak{Z}\right\} \in C_{\text{top}}^{2p-n}(\x;\QQ),
\]
and according to \cite{KLM}\begin{equation}\label{eqnResFormula}d\left[R_{\mathfrak{Z}}\right]=\Omega_{\mathfrak{Z}}-(2\pi i)^{n}\delta_{T_{\mathfrak{Z}}}-2\pi iR_{\partial_{B}\mathfrak{Z}}.
\end{equation}In particular, given $f_{1},\ldots,f_{n}\in\mathcal{O}_{\text{alg}}^{*}(U)$
($U\subset\x$ Zariski open) for which the Zariski closure of 
\[
\Gamma_{\underline{f}}:=\left\{ (x,f_{1}(x),\ldots,f_{n}(x))\,|\, x\in U\right\} \subset\x\times\square^{n}
\]
 is a good precycle, each term of $R(f_{1},\ldots,f_{n})$ defines
an $(n-1)$-current on $\x$.

The relative dearth of coordinate hypersurfaces in projective space
makes defining a telescoping sequence of currents a greater challenge
than in the cubical case. Writing $X_{0}:\cdots:X_{n}$ for the projective
coordinates, the closure of $\Gamma_{\left(\frac{X_{1}}{X_{0}},\cdots,\frac{X_{n}}{X_{0}}\right)}$
is not even a precycle. On its own this is not necessarily a problem,
but the non-integrability of $(\log x)\frac{dx}{x}\delta_{T_{x}}$
against $1$ on $D_{\epsilon}(0)$ means that along the hyperplane
at infinity, certain terms of $R\left(\frac{X_{1}}{X_{0}},\cdots,\frac{X_{n}}{X_{0}}\right)$
(for $n\geq3$) fail individually to yield currents on $\mathbb{P}^{n}$.
This must be corrected if any sort of computation or manipulation
is to take place. Moreover, it is not at all clear how to generalize
the construction of a bounding membrane in \cite{Ke1,Ke2}.

To get around the termwise-nonconvergence problem, there are two natural
choices on $\PP^{n}$: \begin{equation} \label{eqnS}S_n^{\Delta}:=S_n^{\Delta}(X_0:\cdots :X_n):=R_n\left(-\frac{X_1}{X_0},-\frac{X_2}{X_1},\ldots ,-\frac{X_n}{X_{n-1}}\right)
\end{equation}and \begin{equation}\label{eqnRsimp}\begin{matrix} R_n^{\Delta}:=R_n^{\Delta}(X_0:\cdots:X_n):=\\ \\ R_n\left( -\frac{(X_1+\cdots +X_n)}{X_0} , -\frac{(X_2+\cdots +X_n)}{X_1},\ldots,-\frac{X_n}{X_{n-1}}\right), \end{matrix}
\end{equation}with $R_{n}(\cdots)$ as in \eqref{eqnRn}, interpreted in the sense
of $\int_{\mathbb{P}^{n}}R_{n}(\cdots)\wedge\omega:=\lim_{\epsilon\to0}\int_{\mathbb{P}^{n}\backslash U_{\epsilon}}R_{n}(\cdots)\wedge\omega$
for $U_{\epsilon}$ a tubular neighborhood of the singular set of
$R_{n}(\cdots)$. The first version can be more convenient for reciprocity
laws, but the second is essential for defining Abel-Jacobi maps, as
we shall discover below.
\begin{lem}
\label{lemma 2.1}Each term of $S_{n}^{\Delta}$ and $R_{n}^{\Delta}$
belongs to $\mathcal{D}^{n-1}(\PP^{n})$.\end{lem}
\begin{proof}
For $S_{n}^{\Delta}$, we remark that no $X_{i}$ appears more than
twice in any term, and that occurrences are always adjacent. This
produces singularities of the form $(\log z)\delta_{T_{-z}}$ and
$(\log z)\frac{dz}{z}$ (which are integrable against smooth forms),
and exterior products of such, while prohibiting $\frac{dz}{z}\delta_{T_{-z}}$
and $(\log z)\frac{dz}{z}\delta_{T_{-z}}$.

While $R_{n}^{\Delta}$ appears to be more complicated, it turns out
to be even better behaved. Consider the cycle\begin{equation}\label{*S2}
\Gamma_{\underline{f}^{[n]}}\in Z^{n}(\PP^{n}\times\square^{n})
\end{equation}where\small 
\[
\underline{f}^{[n]}:=\underline{f}[X_{0}:\cdots:X_{n}]:=\left(-\frac{X_{1}+\cdots+X_{n}}{X_{0}},-\frac{X_{2}+\cdots+X_{n}}{X_{1}},\ldots,-\frac{X_{n}}{X_{n-1}}\right).
\]
\normalsize Its intersections with $\PP^{n}\times\imath_{j}^{\infty}(\square^{n-1})$
($j=1,\ldots,n$) and $\PP^{n}\times\imath_{n}^{0}(\square^{n-1})$
are $\Gamma_{\underline{f}[X_{0}:\cdots:\widehat{X_{k}}:\cdots:X_{n}]}$
for some $k$. Those with $\PP^{n}\times\imath_{j}^{0}(\square^{n-1})$
for $j=1,\ldots,n-1$ are concentrated over the loci $\PP^{j-1}\cong\{X_{j}=\cdots=X_{n}=0\}\overset{I_{j}}{\hookrightarrow}\PP^{n}$,
since away from this set, $X_{j}+\cdots+X_{n}=0$ $\implies$ one
of $-\frac{X_{j+1}+\cdots+X_{n}}{X_{j}},\,\ldots,\,-\frac{X_{n}}{X_{n-1}}$
equals $1$. In fact these intersections are proper and yield degenerate
cycles, of the form $\left(I_{j}\times\imath_{j}^{0}\right)_{*}p_{j,\ldots,n}^{*}\Gamma_{\underline{f}^{[j-1]}}.$
We conclude that $\Gamma_{\underline{f}^{[n]}}$ is a precycle.

Moreover, writing $x_{i}:=\frac{X_{i}}{X_{0}}$, it meets $(\CC^{*})^{n}\times T_{z_{1}}$,
$(\CC^{*})^{n}\times(T_{z_{1}}\cap T_{z_{2}})$, $\ldots$, $(\CC^{*})^{n}\times T_{n}$
over the subsets of $(\CC^{*})^{n}$ defined by: $(x_{1}+\cdots+x_{n})>0$;
$x_{1},(x_{2}+\cdots+x_{n})>0$; $x_{1},x_{2},(x_{3}+\cdots+x_{n})>0$;
$\ldots$; $x_{1},\ldots,x_{n}>0$. The intersections with $\Gamma_{\underline{f}[X_{0}:\cdots:\widehat{X_{k}}:\cdots:X_{n}]}$
in $(\CC^{*})^{n-1}\subset\{X_{k}=0\}$ behave similarly, and so we
conclude that $\Gamma_{\underline{f}^{[n]}}\in Z_{\RR}^{n}(\PP^{n},n)$.
It follows immediately that (term for term) $R_{n}^{\Delta}=R_{\Gamma_{\underline{f}^{[n]}}}$
is a current.
\end{proof}
To complete either \eqref{eqnS} or \eqref{eqnRsimp} to a triple,
the currents \begin{equation}\Omega_n^{\Delta}:=\frac{dx_1}{x_1}\wedge\cdots\wedge \frac{dx_n}{x_n} \; ,\;\;\;T_n^{\Delta}:=\overline{\left(\RR_{>0}\right)^{\times n}}. 
\end{equation}on $\PP^{n}$ will be needed.
\begin{lem}
\label{lem 2.2}We have $T_{n}^{\Delta}=T_{n}\left(-\frac{X_{1}}{X_{0}},\ldots,-\frac{X_{n}}{X_{n-1}}\right)=T_{n}\left(\underline{f}^{[n]}\right)$
and $\Omega_{n}^{\Delta}=\Omega_{n}\left(-\frac{X_{1}}{X_{0}},-\frac{X_{2}}{X_{1}},\ldots,-\frac{X_{n}}{X_{n-1}}\right)=\Omega_{n}\left(\underline{f}^{[n]}\right)$.\end{lem}
\begin{proof}
For the chains the first equality is clear, and $T_{n}^{\Delta}=T_{n}(\underline{f}^{[n]})$
follows from the proof of \ref{lemma 2.1}. To illustrate the latter
point for $n=3$: in $T_{3}(\underline{f}^{[3]})=T_{-(x_{1}+x_{2}+x_{3})}\cap T_{-\frac{(x_{2}+x_{3})}{x_{1}}}\cap T_{-\frac{x_{3}}{x_{2}}}$,
we have $x_{1}+x_{2}+x_{3}=a$, $x_{2}+x_{3}=bx_{1}$, and $x_{3}=cx_{2}$
where $a,b,c>0$. Hence $x_{1}=\frac{a}{1+b}>0$, $x_{2}=\frac{b}{1+c}x_{1}>0$,
and $x_{3}=cx_{2}>0$. The reverse inclusion is clear.

The two equalities of $(n,0)$-currents follows, via the $\text{dlog}$
map from symbols to forms, from the following computation in Milnor
$K$-theory of $\CC(x_{1},\ldots,x_{n})$:\small
\[
\left\{ -\frac{(X_{1}+\cdots+X_{n})}{X_{0}},-\frac{(X_{2}+\cdots+X_{n})}{X_{1}},\ldots,-\frac{(X_{n-1}+X_{n})}{X_{n-2}},-\frac{X_{n}}{X_{n-1}}\right\} =
\]
\[
\left\{ -\frac{X_{1}}{X_{0}}\left(1+\frac{X_{2}}{X_{1}}\left(1+\frac{X_{3}}{X_{2}}\left(1+\cdots\right)\right)\right),-\frac{X_{2}}{X_{1}}\left(1+\frac{X_{3}}{X_{2}}\left(1+\cdots\right)\right),\ldots\right.\mspace{50mu}
\]
\[
\mspace{280mu}\left.\ldots,-\frac{X_{n-1}}{X_{n-2}}\left(1+\frac{X_{n}}{X_{n-1}}\right),-\frac{X_{n}}{X_{n-1}}\right\} =
\]
\[
\left\{ -\frac{X_{1}}{X_{0}},-\frac{X_{2}}{X_{1}}\left(1+\frac{X_{3}}{X_{2}}\left(1+\cdots\right)\right),-\frac{X_{3}}{X_{2}}\left(1+\cdots\right),\ldots,-\frac{X_{n}}{X_{n-1}}\right\} =\cdots=
\]
\[
\left\{ -\frac{X_{1}}{X_{0}},-\frac{X_{2}}{X_{1}},-\frac{X_{3}}{X_{2}},\ldots,-\frac{X_{n}}{X_{n-1}}\right\} =\cdots=\left\{ -\frac{X_{1}}{X_{0}},-\frac{X_{2}}{X_{0}},\ldots,-\frac{X_{n}}{X_{0}}\right\} ,
\]
\normalsize where we have used in particular the relations $\{\ldots,a,\ldots,-a,\ldots\}=1=\{\ldots,a,\ldots,1-a,\ldots\}$.
\end{proof}
This brings us to the main point. Writing $\rho_{j}\left[\xi_{0}:\cdots:\xi_{n-1}\right]:=\left[\xi_{0}:\cdots:\xi_{j-1}:0:\xi_{j}:\cdots:\xi_{n-1}\right]$
for the inclusion of $\{X_{j}=0\}$ in $\PP^{n}$, we have
\begin{prop}
\label{prop telescope}Let $\mathfrak{R}_{n}$ stand for $R_{n}^{\Delta}$
or $S_{n}^{\Delta}$. Then $d\left[\mathfrak{R}_{n}\right]=\Omega_{n}^{\Delta}-(2\pi i)^{n}\delta_{T_{n}^{\Delta}}+2\pi i\sum_{j=0}^{n}(-1)^{j}(\rho_{j})_{*}\mathfrak{R}_{n-1}.$\end{prop}
\begin{proof}
The computation of $\Gamma_{\underline{f}^{[n]}}\cdot\left(\PP^{n}\times\imath_{k}^{\epsilon}\left(\square^{n-1}\right)\right)$
in the proof of Lemma \ref{lemma 2.1} implies 
\[
\partial_{B}\Gamma_{\underline{f}^{[n]}}=\sum(-1)^{j}(\rho_{j})_{*}\Gamma_{\underline{f}^{[n-1]}},
\]
which together with \eqref{eqnResFormula} gives the result for $R_{n}^{\Delta}$.

For $S_{n}^{\Delta}$, the correct residues are suggested by the corresponding
tame symbols in Milnor $K$-theory:\tiny
\[
\left\{ -\frac{X_{1}}{X_{0}},\ldots,-\frac{X_{\ell}}{X_{\ell-1}},-\frac{X_{\ell+1}}{X_{\ell}},\ldots-\frac{X_{n}}{X_{n-1}}\right\} =\left\{ -\frac{X_{1}}{X_{0}},\ldots,-\frac{X_{\ell}}{X_{\ell-1}},-\frac{X_{\ell+1}}{X_{\ell-1}},\ldots,-\frac{X_{n}}{X_{n-1}}\right\} 
\]
\small
\[
\overset{Tame_{(X_{\ell})}}{\longmapsto}\left\{ -\frac{X_{1}}{X_{0}},\ldots,-\frac{X_{\ell+1}}{X_{\ell-1}},\ldots,-\frac{X_{n}}{X_{n-1}}\right\} .
\]
\normalsize Since $\Gamma_{\left(-\frac{X_{1}}{X_{0}},\ldots,-\frac{X_{n}}{X_{n-1}}\right)}$
isn't a precycle, we must compute explicitly: $S_{n}^{\Delta}\left(X_{0}:\cdots:X_{n}\right)=$\small
\[
S_{\ell-1}^{\Delta}\left(X_{0}:\ldots:X_{\ell-1}\right)\wedge\Omega_{n-\ell+1}^{\Delta}\left(X_{\ell-1}:X_{\ell}:\cdots:X_{n}\right)
\]
\[
+(-2\pi i)^{\ell-1}\delta_{T_{\ell-1}^{\Delta}\left(X_{0}:\cdots:X_{\ell-1}\right)}\cdot S_{2}^{\Delta}\left(X_{\ell-1}:X_{\ell}:X_{\ell+1}\right)\wedge\Omega_{n-\ell-1}^{\Delta}\left(X_{\ell+1}:\cdots:X_{n}\right)
\]
\[
+(-2\pi i)^{\ell+1}\delta_{T_{\ell+1}^{\Delta}\left(X_{0}:\cdots:X_{\ell}:X_{\ell+1}\right)}\cdot S_{n-\ell-1}^{\Delta}\left(X_{\ell+1}:\cdots:X_{n}\right),
\]
\normalsize $(-1)^{\ell}\text{Res}_{(X_{\ell})}$ of which is\small
\[
S_{\ell-1}^{\Delta}\left(X_{0}:\cdots:X_{\ell-1}\right)\wedge\Omega_{n-\ell}^{\Delta}\left(X_{\ell-1}:X_{\ell+1}:\cdots:X_{n}\right)
\]
\[
+(-2\pi i)^{\ell-1}\delta_{T_{\ell-1}^{\Delta}\left(X_{0}:\cdots:X_{\ell-1}\right)}\cdot\log\left(-\frac{X_{\ell+1}}{X_{\ell-1}}\right)\Omega_{n-\ell-1}^{\Delta}\left(X_{\ell+1}:\cdots:X_{n}\right)
\]
\[
+(-2\pi i)^{\ell}\delta_{T_{\ell}^{\Delta}\left(X_{0}:\cdots:X_{\ell-1}:X_{\ell+1}\right)}\cdot S_{n-\ell-1}^{\Delta}\left(X_{\ell+1}:\cdots:X_{n}\right).
\]
\normalsize $=S_{n-1}^{\Delta}\left(X_{0}:\cdots:\widehat{X_{\ell}}:\cdots:X_{n}\right)$.
Here the residue of $\Omega_{n-\ell+1}^{\Delta}$ follows from the
$K$-theory computation, and the boundary of $T_{\ell+1}^{\Delta}(X_{0}:\cdots:X_{\ell+1})$
from the fact that it is just the closure of $\{x_{1},\ldots,x_{\ell+1}\in\RR_{>0}\}$.
Finally, using the fact that $\log\left(\frac{f}{g}\right)=\log(f)-\log(g)$
where $g\in\RR_{>0}$, we have $S_{2}^{\Delta}\left(X_{0}:X_{1}:X_{2}\right)=$\small
\[
\log\left(-x_{1}\right)\frac{dx_{2}}{x_{2}}+\left\{ -\log\left(-x_{1}\right)\frac{dx_{1}}{x_{1}}+2\pi i\log\left(x_{1}\right)\delta_{T_{-x_{1}}}\right\} -2\pi i\log\left(-x_{2}\right)\delta_{T_{-x_{1}}}.
\]
\normalsize One checks that $d$ of the bracketed current is zero,
and so the only contribution  to $\text{Res}_{(X_{1})}$ (namely,
$-\log\left(-x_{2}\right)$) comes from the last term.
\end{proof}
For a dose of concreteness, here is a simple computation involving
$S_{5}^{\Delta}$.
\begin{example}
Let $\mathfrak{Z}\subset\PP^{5}$ be the $\PP^{2}$ obtained by projectivizing
the row space of
\[
\left(\begin{array}{cccccc}
-1 & 0 & 1 & 0 & 0 & 1\\
0 & 1 & -1 & 1 & 0 & 0\\
a & -1 & 0 & 0 & 1 & 0
\end{array}\right),
\]
with coordinates $[Y_{0}:Y_{1}:Y_{2}]$. Writing $z:=\frac{Y_{1}}{Y_{0}}$,
$w:=\frac{Y_{2}}{Y_{0}}$, the pullback of $R_{5}\left(-\frac{X_{1}}{X_{0}},-\frac{X_{2}}{X_{1}},-\frac{X_{3}}{X_{2}},-\frac{X_{4}}{X_{3}},-\frac{X_{5}}{X_{4}}\right)$
to $\mathfrak{Z}$ takes the form 
\[
R_{5}\left(\frac{z-w}{1-aw},\frac{z-1}{z-w},\frac{z}{z-1},-\frac{w}{z},-\frac{1}{w}\right).
\]
$T_{\frac{z-w}{1-aw}}\cap T_{\frac{z-1}{z-w}}$ is a triangular membrane
bounding on $z=w$, $z=1$ and $w=\frac{1}{a}$, and for $a\notin[1,\infty)$
we have $T_{\frac{z-w}{1-aw}}\cap T_{\frac{z-1}{z-w}}\cap T_{\frac{z}{z-1}}=\emptyset$.
Hence the last two terms are zero on $\mathfrak{Z}$, as are the first
two by Hodge type, and 
\[
\frac{1}{(2\pi i)^{2}}\int_{\mathfrak{Z}}S_{5}^{\Delta}=\int_{w=1}^{\frac{1}{a}}\int_{z=1}^{w}\log\left(\frac{z}{z-1}\right)\text{dlog}\left(\frac{w}{z}\right)\wedge\text{dlog}\left(\frac{1}{w}\right).
\]
Substituting $u=\frac{1}{z}$, $v=\frac{1}{w}$, the above 
\[
=-\int_{v=1}^{a}\int_{u=1}^{v}\log(1-u)\text{dlog}(u)\wedge\text{dlog}(v)
\]
\[
=\int_{v=1}^{a}\left(Li_{2}(v)-Li_{2}(1)\right)\text{dlog}(v)
\]
\[
=Li_{3}(a)-Li_{3}(1)-\log(a)Li_{2}(1)
\]
\[
=Li_{3}(a)-\zeta(3)-\frac{\pi^{2}}{6}\log(a).
\]

\end{example}

\section{Abel-Jacobi maps for simplicial higher Chow groups}
\label{sec:abel-jacobi-maps}

Let $\x$ be a smooth projective variety. The complex 
\[
C_{\mathcal{D}}^{\bullet}(\x;\QQ(p)):=C_{\text{top}}^{\bullet}\left(\x;(2\pi i)^{p}\QQ\right)\oplus F^{p}\mathcal{D}^{\bullet}(\x)\oplus\mathcal{D}^{\bullet-1}(\x)
\]
of abelian groups with differential
\[
D\left(T,\Omega,R\right):=\left(-\partial T,-d[\Omega],d[R]-\Omega+\delta_{T}\right)
\]
computes the Deligne cohomology 
\[
H_{\mathcal{D}}^{*}\left(\x;\QQ(p)\right):=H^{*}\left\{ C_{\mathcal{D}}^{\bullet}\left(\x;\QQ(p)\right)\right\} .
\]
These latter spaces are the targets for the $AJ$ (rational regulator)
maps, whose explicit construction on the \emph{simplicial} higher
Chow complex is the subject of this section.

The idea is to replace $(\square^{n},\partial\square^{n})$ in the
KLM-construction by 
\[
(\Delta^{n},\partial\Delta^{n}):=\left(\PP^{n}\backslash H_{n},\bigcup_{j=0}^{n}\rho_{j}\left(\PP^{n-1}\backslash H_{n-1}\right)\right)
\]
where $H_{n}$ is the special hyperplane cut out by $X_{0}+\cdots+X_{n}=0$.
We then define precycles (resp. good precycles)
\[
Z_{\Delta,\RR}^{p}(\x,n)\subset Z_{\Delta}^{p}(\x,n)\subset Z^{p}(\x\times\Delta^{n})
\]
 to be those cycles meeting arbitrary intersections of the $\x\times\rho_{j}(\Delta^{n-1})$
(resp. of these and the $T_{\frac{X_{1}+\cdots+X_{n}}{-X_{0}}}$,
$T_{\frac{X_{1}+\cdots+X_{n}}{-X_{0}}}\cap T_{\frac{X_{2}+\cdots+X_{n}}{-X_{1}}}$,
etc.) properly. The Bloch boundary map\begin{equation}\label{boundary}\partial_{\mathcal{B}}\mathfrak{Z}:=\sum_{j=0}^{n}(-1)^{j}\rho_{j}^{*}\mathfrak{Z}
\end{equation}makes these into quasi-isomorphic complexes:
\begin{prop}
\label{ML1}$H_{n}\left(Z_{\Delta,\RR}^{p}(\x,\bullet)\right)\cong H_{n}\left(Z_{\Delta}^{p}(\x,\bullet)\right)\cong CH^{p}(\x,n)$.%
\footnote{The proof is deferred to the first Appendix to this section so as
not to interrupt the main flow of ideas.%
}
\end{prop}
We shall define a morphism \begin{equation}Z^p_{\Delta,\RR}(\x,-\bullet) \overset{\widetilde{AJ}_{\Delta,\x}^{p,-\bullet}}{\longrightarrow} C_{\mathcal{D}}^{2p+\bullet}(\x;\QQ(p)),
\end{equation} which then automatically induces (simplicial) $AJ$ maps \begin{equation}\label{eqn 3.3}AJ^{p,n}_{\Delta,\x}:CH^p(\x,n)\longrightarrow H^{2p-n}_{\mathcal{D}}(\x,\QQ(p)).
\end{equation}Namely, writing $\pi_{\Delta}$, $\pi_{\x}$ for projections from
the desingularization $\widetilde{\mathfrak{Z}}$ to $\Delta^{n}$,
$\x$, we set
\[
R_{\mathfrak{Z}}^{\Delta}:=(\pi_{\x})_{*}(\pi_{\Delta})^{*}R_{n}^{\Delta}\;,\;\;\;\;\;\Omega_{\mathfrak{Z}}^{\Delta}:=(\pi_{\x})_{*}(\pi_{\Delta})^{*}\Omega_{n}^{\Delta}\;,
\]
\[
T_{\mathfrak{Z}}^{\Delta}:=(\pi_{\x})_{*}\left\{ (\x\times T_{n}^{\Delta})\cap\mathfrak{Z}\right\} \;,
\]
and \begin{equation}\label{AJdef}\widetilde{AJ}^{p,n}_{\Delta,\x}(\mathfrak{Z}):=(2\pi i)^{p-n}\left((2\pi i)^n T^{\Delta}_{\mathfrak{Z}},\Omega^{\Delta}_{\mathfrak{Z}},R^{\Delta}_{\mathfrak{Z}}\right).
\end{equation} 
\begin{thm}
\label{AJ Thm}$\widetilde{AJ}_{\Delta,\x}^{p,-\bullet}$ is a well-defined
morphism of complexes. The induced maps $AJ_{\Delta,\x}^{p,n}$ recover
Bloch's cycle-class maps (in the sense of Definition \ref{Def BCCM}
below).
\end{thm}
The proof is simple but somewhat formal, and so we shall preface it
with a (probably more helpful) direct argument that \eqref{eqn 3.3}
induces a map of complexes. First there is the question of whether
it is well-defined, which splits into ``algebraic'' and ``analytic''
parts. The latter issue, of whether $R_{\mathfrak{Z}}^{\Delta}$ and
$\Omega_{\mathfrak{Z}}^{\Delta}$ are actually in $D^{2p-n-1}(\x)$
resp. $F^{p}D^{2p-n}(\x)$ (since pullbacks $\pi_{\Delta}^{*}$ need
not preserve currents), is implicitly resolved in the proof below
(by the relation to the cubical KLM currents). For reference, we have
also included an explicit argument that $R_{\mathfrak{Z}}^{\Delta}$
is a current in the second appendix to this section.

Now $Z^{p}(\x\times\Delta^{n})=Z^{p}(\x\times\PP^{n})/Z^{p}(\x\times H_{n}),$
and the ``algebraic'' well-definedness refers to the requirement
that $\widetilde{AJ}_{\Delta,\x}^{p,n}$ vanish on admissible precycles
with support in $\x\times H_{n}$. In fact it suffices to check the
following, writing $\overline{\mathfrak{Z}}\in Z^{p}(\x\times\PP^{n+1})$
for the closure of $\mathfrak{Z}\in Z^{p}(\x\times\Delta^{n+1})$.
Given $j\in\{0,\ldots,n+1\}$, $\mathfrak{Z}\in Z_{\Delta,\RR}^{p}(\x,n+1)$,
and $\mathfrak{W}$ an irreducible component of $\rho_{j}^{*}\overline{\mathfrak{Z}}$
sitting inside $\x\times H_{n}$, we must have $R_{\mathfrak{W}}^{\Delta},\Omega_{\mathfrak{W}}^{\Delta},T_{\mathfrak{W}}^{\Delta}$
all zero. But on $H_{n}$ we have $X_{0}+\cdots+X_{n}\equiv0$, and
on $\mathfrak{W}$ we cannot have all $X_{i}\equiv0$. If (say) $X_{0}|_{\mathfrak{W}}\equiv\cdots\equiv X_{k-1}|_{\mathfrak{W}}\equiv0$
but $X_{k}$ is not identically zero, then $-\frac{X_{k+1}+\cdots+X_{n}}{X_{k}}|_{\mathfrak{W}}\equiv1$
and the currents are trivial as desired.

To verify that $\widetilde{AJ}_{\Delta,\x}^{p,-\bullet}$ is a morphism
of complexes, we use the formula in Proposition \ref{prop telescope}.
This gives for each $\mathfrak{Z}\in Z_{\Delta,\RR}^{p}(\x,n)$
\[
d[R_{\mathfrak{Z}}^{\Delta}]=\Omega_{\mathfrak{Z}}^{\Delta}-(2\pi i)^{n}\delta_{T_{\mathfrak{Z}}^{\Delta}}+2\pi i\sum_{j=0}^{n}(-1)^{j}R_{\rho_{j}^{*}\mathfrak{Z}}^{\Delta}
\]
\begin{equation}\label{eqn dR}=\Omega_{\mathfrak{Z}}^{\Delta}-(2\pi i)^{n}\delta_{T_{\mathfrak{Z}}^{\Delta}}+2\pi iR_{\partial_{\mathcal{B}}\mathfrak{Z}}^{\Delta},
\end{equation}while 
\[
\left\{ \begin{array}{c}
\partial T_{n}^{\Delta}=\sum_{j=0}^{n}(-1)^{j-1}\rho_{j}\left(T_{n-1}^{\Delta}\right)\\
d[\Omega_{n}^{\Delta}]=2\pi i\sum_{j=0}^{n}(-1)^{j-1}(\rho_{j})_{*}\Omega_{n-1}^{\Delta}
\end{array}\right.\;\;\implies
\]
\[
\left\{ \begin{array}{c}
\partial T_{\mathfrak{Z}}^{\Delta}=-T_{\partial_{\mathcal{B}}\mathfrak{Z}}^{\Delta}\\
d[\Omega_{\mathfrak{Z}}^{\Delta}]=-2\pi i\Omega_{\partial_{\mathcal{B}}\mathfrak{Z}}^{\Delta}
\end{array}\right.,
\]
and so
\[
D\left((2\pi i)^{n}T_{\mathfrak{Z}}^{\Delta},\Omega_{\mathfrak{Z}}^{\Delta},R_{\mathfrak{Z}}^{\Delta}\right)=2\pi i\left((2\pi i)^{n-1}T_{\partial_{\mathcal{B}}\mathfrak{Z}}^{\Delta},\Omega_{\partial_{\mathcal{B}}\mathfrak{Z}}^{\Delta},R_{\partial_{\mathcal{B}}\mathfrak{Z}}^{\Delta}\right)
\]
which yields $D\circ\widetilde{AJ}=\widetilde{AJ}\circ\partial_{\mathcal{B}}$
as needed.

As mentioned in the Introduction, the simplicial $AJ$ formula will
be particularly natural for linear higher Chow cycles derived from
elements of $H_{2n-1}(GL_{n}(K),\ZZ)$ ($K$ a number field). While
we won't pursue this application in the present paper, here is an
example of what this will look like on an irreducible component of
such a cycle.
\begin{example}
Let $\alpha\in\CC\backslash\RR_{\leq0}$, and consider the linear
precycle $\mathfrak{Z}:=$ 
\[
\left\{ [\alpha Z-W:Z:-Z:W]\,|\,[Z:W]\in\PP^{1}(\mathbb{C})\right\} \in Z_{\Delta,\RR}^{2}(\text{Spec }\CC,3),
\]
for $\alpha\in\CC\backslash[1,\infty)$. It has boundary $\partial\mathfrak{Z}=[1:-1:\alpha]-[\alpha:1:-1],$
and should be thought of as a simplicial analogue of the Totaro (pre)cycle. 

Writing $z:=\frac{Z}{W}$ for the coordinate on $\PP^{1}$, $R_{\mathfrak{Z}}^{\Delta}\in Ext_{\text{MHS}}^{1}(\QQ(0),\QQ(2))\cong\CC/\QQ(2)$
is computed by
\[
\int_{\mathfrak{Z}}R_{3}^{\Delta}=\int_{\PP^{1}}R\left(-\frac{W}{\alpha Z-W},-\frac{W-Z}{Z},-\frac{W}{-Z}\right)
\]
\[
=\int_{\PP^{1}}R\left(\frac{1}{1-\alpha z},1-\frac{1}{z},\frac{1}{z}\right).
\]
Since $T_{\frac{1}{1-\alpha z}}=\left(\tfrac{1}{\alpha},\infty\right):=\left\{ \tfrac{r}{\alpha}\left|r\in\RR_{>0}\right.\right\} $
(oriented from $\tfrac{1}{\alpha}$ to $\infty$) and $T_{\frac{1}{1-\alpha z}}\cap T_{1-\frac{1}{z}}=(\tfrac{1}{\alpha},\infty)\cap(0,1)=\emptyset$,
this
\[
=\int_{T_{\frac{1}{1-\alpha z}}}\log\left(1-\tfrac{1}{z}\right)\text{dlog}\left(\tfrac{1}{z}\right)
\]
\[
=Li_{2}(\alpha).
\]
\end{example}
\begin{rem}
\label{rem Go}The currents $S_{n}^{\Delta}$ are closer than the
$R_{n}^{\Delta}$ to being invariant with respect to scaling the coordinates,
which apparently makes them more suitable for studying reciprocity
laws and functional equations of polylogarithms. However, they fail
to yield well-defined $AJ$ maps, as they do not vanish on $H_{n}$. 

The real $(n-1)$-currents $r_{n}$ of \cite{Go1} more dramatically
illustrate the problem, as they are actually invariant under scaling
of coordinates, and are prevented \emph{by this property} from vanishing
on $H_{n}$, and hence from defining simplicial $AJ$ maps as claimed
in {[}op. cit.{]}. That they do nevertheless produce $AJ$ maps on
the cubical level, coinciding with the real or imaginary part of Bloch's
invariants, was checked in \cite[sec. 3.1.1]{Ke1}.

It is instructive to demonstrate the issue for $r_{3}$. Let $a\in\CC\backslash\RR$.
According to \cite[Thm. 3.6]{Go1}, 
\[
\int_{\PP^{1}}r_{3}\left(z,1-z,z-a\right)=\frac{1}{2\pi i}D_{2}(a)
\]
where $D_{2}$ is the Bloch-Wigner function.%
\footnote{In fact, for our purposes it suffices to know that the integral is
nonzero. This reduces to nonvanishing of $\int_{\PP^{1}}\log|z-a|\text{dlog}|z|\wedge\text{dlog}|1-z|$,
which follows from that of $\int_{\RR^{2}}\tfrac{y\log|z-a|}{|z|^{2}|1-z|^{2}}dA$
for $a\notin\RR$.%
} By \cite[Prop. 3.2]{Go1}, this
\[
=\int_{\PP^{1}}r_{3}\left(z,2-2z,z-a\right)=\int_{\PP^{1}}r_{3}\left(\frac{z}{a-2},\frac{2-2z}{a-2},\frac{z-a}{a-2}\right)
\]
\[
=\int_{\mathfrak{Z}_{a}}r_{3}(x_{1},x_{2},x_{3}),
\]
where $\mathfrak{Z}_{a}=\left\{ [a-2:z:2-2z:z-a]\,|\, z\in\PP^{1}\right\} .$
But $(a-2)+z+(2-2z)+(z-a)\equiv0$ $\implies$ $\mathfrak{Z}_{a}\subset H_{3}$
$\implies$ $\mathfrak{Z}_{a}=0\in Z^{2}(\CC,3)$. So $\mathfrak{Z}\mapsto\int_{\mathfrak{Z}}r_{3}$
does not induce a well-defined map $\tilde{r}:\, Z^{2}(\CC,3)\to\RR(1)$.

If one tries to make $\tilde{r}$ well-defined by insisting that it
be ``zero on zero'', another problem emerges: we do not obtain a
map of complexes
\[
\begin{array}{ccccccc}
\cdots\to & Z^{2}(\CC,4) & \to & Z^{2}(\CC,3) & \to & Z^{2}(\CC,2) & \to\cdots\\
 & \downarrow &  & \;\downarrow\tilde{r} &  & \downarrow\\
\cdots\to & 0 & \to & \RR(1) & \to & 0 & \to\cdots.
\end{array}
\]
If we take $\tilde{\mathfrak{Z}}:=$
\[
\left\{ [V:(a-2)(W+V):Z-V:2(W-Z):Z-aW]\right\} _{[V:W:Z]\in\PP^{1}}
\]
in $Z^{2}(\CC,4)$, then $\partial\tilde{\mathfrak{Z}}=\sum_{j=1}^{4}(-1)^{j}\rho_{j}^{*}\tilde{\mathfrak{Z}}$
since $\rho_{0}^{*}\tilde{\mathfrak{Z}}=\mathfrak{Z}_{a}=0$. But
the reciprocity properties of $r_{3}$ imply $\sum_{j=0}^{4}(-1)^{j}\int_{\rho_{j}^{*}\tilde{\mathfrak{Z}}}r_{3}=0$,
which gives 
\[
\tilde{r}(\partial\tilde{\mathfrak{Z}})=\sum_{j=1}^{4}(-1)^{j}\int_{\rho_{j}^{*}\tilde{\mathfrak{Z}}}r_{3}=-\int_{\rho_{0}^{*}\tilde{\mathfrak{Z}}}r_{3}=-\tfrac{1}{2\pi i}D_{2}(a)\neq0.
\]
So apparently, the only way to fix the problem is to replace $r_{3}=r_{3}\left(\tfrac{X_{1}}{X_{0}},\tfrac{X_{2}}{X_{0}},\tfrac{X_{3}}{X_{0}}\right)$
by something like $r_{3}\left(-\tfrac{X_{1}+X_{2}+X_{3}}{X_{0}},-\tfrac{X_{2}+X_{3}}{X_{1}},-\tfrac{X_{3}}{X_{2}}\right)$,
which affects its properties and calls into question (for example)
the known proof%
\footnote{cf. Prop. 16 in \cite{dJ}; we do expect that this can be fixed.%
} that linear higher Chow groups of number fields surject onto the
usual higher Chow cycles.
\end{rem}

\subsection*{Proof of Theorem \ref{AJ Thm}}

We shall need the subcomplexes of \emph{normalized} precycles 
\[
N_{\Delta}^{p}(\x,n):=\bigcap_{j=0}^{n-1}\ker\rho_{j}^{*}\subset Z_{\Delta}^{p}(\x,n)
\]
\[
N_{\Delta,\RR}^{p}(\x,n):=N_{\Delta}^{p}(\x,n)\cap Z_{\Delta,\RR}^{p}(\x,n)
\]
and the following ``moving lemma'' (verified in the first appendix
to this section):
\begin{prop}
\label{ML2}$H_{n}(N_{\Delta,\RR}^{p}(\x,\bullet))\cong H_{n}(N_{\Delta}^{p}(\x,\bullet))\cong H_{n}(Z_{\Delta}^{p}(\x,\bullet)).$
\end{prop}
With this, we may define the Bloch cycle-class map:
\begin{defn}
\label{Def BCCM} Let $\xi\in CH^{p}(\x,n)$ ($n\geq1$) have normalized
representative $\mathfrak{Z}\in\ker(\partial_{\mathcal{B}})\subset N_{\Delta}^{p}(\x,n)$;
that is, all $\rho_{j}^{*}\mathfrak{Z}=0$. Denoting $\x\times\Delta^{n}=:\Delta_{\x}^{n}$,
etc., the localization sequence for
\[
(U^{\Delta},\partial U^{\Delta}):=\left(\Delta_{\x}^{n}\backslash\left|\mathfrak{Z}\right|,\partial\Delta_{\x}^{n}\backslash\left\{ \left|\mathfrak{Z}\right|\cap\partial\Delta_{\x}^{n}\right\} \right)
\]
leads to an extension (with $\QQ(p)$-coefficients) \small \begin{equation}\label{extension}\xymatrix{H^{2p-n-1}(\x) \ar @{=} [d] \ar @{^(->} [r] & \mathbb{E}^{\Delta} \ar @{^(->} [d] \ar @{->>} [r] & \QQ(-p)= \langle \mathfrak{Z} \rangle \ar @{^(->} [d]
\\
H^{2p-1}\left( \Delta^n_{\x},\partial \Delta^n_{\x} \right) \ar @{^(->} [r] & H^{2p-1}(U^{\Delta},\partial U^{\Delta}) \ar @{->>} [r] & H^{2p}_{|\mathfrak{Z}|}\left( \Delta^n_{\x},\partial \Delta^n_{\x} \right)^{\circ} }
\end{equation}\normalsize We \emph{define} 
\[
c_{\mathcal{B}}(\xi)\in Ext_{\text{MHS}}^{1}\left(\QQ(0),H^{2p-n-1}(\x,\QQ(p))\right)\cong H_{\mathcal{D}}^{2p-n}(\x,\QQ(p))
\]
to be the extension class of the top sequence.
\end{defn}
The proof of the Theorem will now proceed in the three steps:\vspace{2mm}

$\underline{\text{Step 1}}:$ \emph{The cube-to-simplex map.} Recall
that $\Gamma_{\underline{f}^{[n]}}$ (cf. \eqref{*S2}) is the restriction
to $\PP^{n}\times\square^{n}$ of the correspondence in $\PP^{n}\times(\PP^{1})^{n}$
given by\begin{equation}\label{eqn matrix}\underset{A(\underline{\lambda},\underline{\sigma})}{\underbrace{\left[\begin{array}{cccccc} \lambda_{1} & \sigma_{1} & \sigma_{1} & \cdots & \sigma_{1} & \sigma_{1}\\ 0 & \lambda_{2} & \sigma_{2} & \cdots & \sigma_{2} & \sigma_{2}\\ 0 & 0 & \lambda_{3} & \cdots & \sigma_{3} & \sigma_{3}\\ \vdots & \vdots & \vdots & \ddots & \vdots & \vdots\\ 0 & 0 & 0 & \cdots & \lambda_{n} & \sigma_{n} \end{array}\right]}}\left[\begin{array}{c} X_{0}\\ X_{1}\\ \vdots\\ X_{n} \end{array}\right]=\left[\begin{array}{c} 0\\ 0\\ \vdots\\ 0 \end{array}\right]
\end{equation}where $[\sigma_{j}:\lambda_{j}]$ are projective coordinates on $\PP_{z_{j}}^{1}$
($z_{j}=\tfrac{\lambda_{j}}{\sigma_{j}}$). Observe that the $(n+1)\times(n+1)$
matrix $B(\underline{\lambda},\underline{\sigma})$, obtained by adding
a row of ones to $A(\underline{\lambda},\underline{\sigma})$, has\begin{equation}\label{eqn det}\det B(\underline{\lambda},\underline{\sigma})=\prod_{j=1}^n (\sigma_j - \lambda_j ).
\end{equation}Now \eqref{eqn det} implies that
\begin{itemize}
\item $A(\underline{\lambda},\underline{\sigma})$ has maximal rank, so
that $\Gamma_{\underline{f}^{[n]}}$ induces a well-defined morphism
from $\square^{n}$ to $\PP^{n}$; and
\item $B(\underline{\lambda},\underline{\sigma})[\underline{X}]=[\underline{0}]$
has no nonzero solution, so that the image of this map avoids the
hyperplane $H_{n}$ (where $X_{0}+\cdots+X_{n}=0$).
\end{itemize}
We shall write
\[
F_{n}:\,\square^{n}\to\Delta^{n}
\]
for this morphism and $\Gamma_{n}\in Z^{n}(\Delta^{n}\times\square^{n})$
for the associated correspondence. The explicit formula\begin{multline}\label{eqn Fn}F_n\left( [\si_1:\la_1],\ldots,[\si_n:\la_n]\right) \; = \; \left[ \si_1 (\si_2-\la_2)\cdots (\si_n-\la_n) :\right.
\\
-\la_1\si_2(\si_3-\la_3)\cdots (\si_n-\la_n) : \la_1\la_2\si_3(\si_4-\la_4)\cdots(\si_n-\la_n) 
\\
\left. :\cdots : (-1)^{n-1}\la_1\cdots \la_{n-1}\si_n : (-1)^n \la_1\cdots\la_n \right]
\end{multline}makes it clear that $F_{n}(\partial\square^{n})\subset\partial\Delta^{n}$.
The induced map 
\[
(F_{n})_{*}:\, H_{n}(\square^{n},\partial\square^{n})\to H_{n}(\Delta^{n},\partial\Delta^{n})
\]
is an isomorphism since it sends representatives $T_{n}\mapsto T_{n}^{\Delta}$,
$\Omega_{n}\mapsto\Omega_{n}^{\Delta}$. (This is essentially the
same computation as in Lemma \ref{lem 2.2}.) Hence for any (smooth
projective) $\x$, denoting $\x\times\square^{n}=:\square_{\x}^{n}$,
etc., 
\[
F_{\x,n}^{*}:\, H^{m}(\Delta_{\x}^{n},\partial\Delta_{\x}^{n})\to H^{m}(\square_{\x}^{n},\partial\square_{\x}^{n})
\]
is an isomorphism for any $m$.\vspace{2mm}

$\underline{\text{Step 2}}:$ \emph{Simplicial to cubical precycles.}
The morphism $F_{n}$ is the composition of an inclusion (of $\square^{n}$
into a larger open subset of $(\PP^{1})^{n}$) with a sequence of
blow-ups at the smooth centers: $X_{1}=\cdots=X_{n}=0$; and (successive
proper transforms of) $X_{2}=\cdots=X_{n}=0$, $\ldots$, $X_{n-1}=X_{n}=0$.
Its positive-dimensional fibers are contained in $\cup_{j=1}^{n-1}\imath_{j}^{0}(\square^{n-1})$
and are degenerate in the sense that one or more $z_{i}$'s (in fact,
$z_{j+1}$ thru $z_{n}$) are arbitrary. For cycles on $\x\times\Delta^{n}$
meeting the blow-up centers properly (which includes $Z_{\Delta}^{p}(\x,n)$),
the pullback under $\text{id}_{X}\times F_{n}:\,\x\times\square^{n}\to\x\times\Delta^{n}$
is well-defined.%
\footnote{That is, the ``preimage'' $(\pi_{13})_{*}(\pi_{12}^{*}\mathfrak{Z}\cdot\pi_{23}^{*}\Gamma_{n})$
of an irreducible $\mathfrak{Z}$ (where $\pi_{ij}$ are the projections
on $\x\times\Delta^{n}\times\square^{n}$) already yields the proper
transform, without having to throw out ``exceptional'' components
contained in $\cup_{j=1}^{n-1}\imath_{j}^{0}(\square^{n-1})$.%
} This yields a map
\[
\begin{array}{cccc}
\Gamma_{\x,n}^{*}: & Z_{\Delta,\RR}^{p}(\x,n) & \to & Z^{p}(\x\times\square^{n})\\
 & \mathfrak{Z} & \mapsto & \mathfrak{Z}^{\square}
\end{array}
\]
which we claim factors through $Z_{\RR}^{p}(\x,n)$.

Let $\tau$ be one of the real chains $T_{z_{1}}\cap\cdots\cap T_{z_{k}}\cap\imath_{J}^{\underline{\epsilon}}(\square^{n-|J|})$,
say of (real) codimension $c$. Inspection of \eqref{eqn matrix}
and \eqref{eqn Fn} shows that $\tau^{\Delta}:=F_{n}(\tau)$ is one
of the real chains $T_{\frac{X_{1}+\cdots+X_{n}}{-X_{0}}}\cap\cdots\cap T_{\frac{X_{\ell}+\cdots+X_{n}}{-X_{\ell-1}}}\cap\rho_{I}(\Delta^{n-|I|})$,
of codimension $c^{\Delta}\geq c$. Since $|\mathfrak{Z}^{\square}|\cap(\x\times\tau)\subset F_{n}^{-1}\left(|\mathfrak{Z}|\cap(\x\times\tau^{\Delta})\right)$,
we have
\[
\dim_{\RR}\left(|\mathfrak{Z}|\cap(\x\times\tau)\right)\leq(c^{\Delta}-c)+\dim_{\RR}\left(|\mathfrak{Z}|\cap(\x\times\tau^{\Delta})\right).
\]
Moreover, as $\mathfrak{Z}\in Z_{\Delta,\RR}^{p}(\x,n)$, we have
$\text{codim}_{\RR}^{|\mathfrak{Z}|}\left(|\mathfrak{Z}|\cap(\x\times\tau^{\Delta})\right)\geq c^{\Delta}$;
it follows that $\text{codim}_{\RR}^{|\mathfrak{Z}^{\square}|}\left(|\mathfrak{Z}^{\square}|\cap(\x\times\tau)\right)\geq c$.
Since $\tau$ was arbitrary, $\mathfrak{Z}^{\square}\in Z_{\RR}^{p}(\x,n)$.

Next we claim that
\[
\Gamma_{\x,\bullet}^{*}:\, Z_{\Delta,\RR}^{p}(\x,\bullet)\to Z_{\RR}^{p}(\x,\bullet)
\]
is a map of complexes, i.e. that \begin{equation} \label{eq !}
\Gamma^*_{\x,n-1}(\partial_{\mathcal{B}}\mathfrak{Z})=\sum_{j=0}^n (-1)^j\Gamma_{n-1}^*\rho_j^*\mathfrak{Z}
\end{equation}and\begin{multline} \label{eq !!}
\partial_{\mathcal{B}}(\Gamma_{\x,n}^*\mathfrak{Z})=\left[-\sum_{j=1}^n (-1)^j (\imath_j^{\infty})^*\mathfrak{Z}^{\square} + (-1)^n (\imath^0_n)^*\mathfrak{Z}^{\square} \right]
\\
+ \sum_{j=1}^{n-1} (-1)^j (\imath^0_j)^* \mathfrak{Z}^{\square} .
\end{multline}agree. Inspection of \eqref{eqn Fn} shows that $F_{n}$ restricts
to $F_{n-1}$ on the facets $\imath_{j}^{\infty}(\square^{n-1})\,\left(\to\rho_{j-1}(\Delta^{n-1})\right)$
(for $j=1,\ldots,n$) and $\imath_{n}^{0}(\square^{n-1})\,\left(\to\rho_{n}(\Delta^{n-1})\right)$,
so that the right-hand side of \eqref{eq !} coincides with the square-bracketed
term in \eqref{eq !!}. The restrictions of $F_{n}$ to the other
facets map $\imath_{j}^{0}(\square^{n-1})\underset{\psi_{j}}{\to}\rho_{\{j,\ldots,n\}}(\Delta^{j-1})$
(with degenerate fibers as mentioned above) for any $j=1,\ldots,n-1$.
Since $\mathfrak{Z}$ meets these $\x\times\rho_{\{j,\ldots,n\}}(\Delta^{j-1})$
properly (in complex codim. $\geq n-j+1$), $\mathfrak{Z}^{\square}$
meets $\x\times\imath_{j}^{0}(\square^{n-1})$ in the $\psi_{j}$-preimage,
which is degenerate. So the remaining terms on the right-hand side
of \eqref{eq !!} are zero in $Z_{\RR}^{p}(\x,n)$.\vspace{2mm}

$\underline{\text{Step 3}}:$ \emph{$AJ_{\Delta}$ and Bloch's map}.
Let 
\[
\widetilde{AJ}_{\x}^{p,-\bullet}:\, Z_{\RR}^{p}(\x,-\bullet)\to C_{\mathcal{D}}^{2p+\bullet}(\x;\QQ(p))
\]
be the map of complexes defined by sending a precycle $\mathcal{W}\in Z_{\RR}^{p}(\x,n)$
to $(2\pi i)^{p-n}\left((2\pi i)^{n}T_{\mathcal{W}},\Omega_{\mathcal{W}},R_{\mathcal{W}}\right)$
\cite{KLM}.%
\footnote{See the beginning of $\S2$ for $T_{\mathcal{W}},\Omega_{\mathcal{W}},R_{\mathcal{W}}$.%
} We claim that the composition
\[
Z_{\RR,\Delta}^{p}(\x,-\bullet)\overset{\Gamma_{\x}^{*}}{\longrightarrow}Z_{\RR}^{p}(\x,-\bullet)\overset{\widetilde{AJ}_{\x}^{p}}{\longrightarrow}C_{\mathcal{D}}^{2p+\bullet}(\x;\QQ(p))
\]
is none other than the $\widetilde{AJ}_{\Delta,\x}^{p}$ of \eqref{AJdef},
proving the first statement of Theorem \ref{AJ Thm}. The point is
that from Lemmas \ref{lemma 2.1}-\ref{lem 2.2} we have $\Gamma_{n}^{*}(T_{n},\Omega_{n},R_{n})=(T_{n}^{\Delta},\Omega_{n}^{\Delta},R_{n}^{\Delta})$
so that \begin{flalign*}
\left( T_n^{\Delta},\Omega_n^{\Delta},R_n^{\Delta}\right) = & \; (\mathfrak{Z}^{\square})^* \left( T_n,\Omega_n,R_n\right)  \\
= & \; {\mathfrak{Z}^* \Gamma^* \left( T_n,\Omega_n,R_n\right) }  \\
= & \; {\mathfrak{Z}^* \left( T_n^{\Delta},\Omega_n^{\Delta},R_n^{\Delta}\right) }  \\
= & \; {\left( T_{\mathfrak{Z}}^{\Delta},\Omega_{\mathfrak{Z}}^{\Delta},R_{\mathfrak{Z}}^{\Delta}\right) } 
\end{flalign*}(where the pullbacks of currents are well-defined by those lemmas
and by \cite{KLM}).

Finally we let $\mathfrak{Z}$ be a normalized (simplicial) precycle
as in Definition \ref{Def BCCM}, with class $\xi$. By the analysis
in Step 2, we have that $\mathfrak{Z}^{\square}:=\Gamma_{\x,n}^{*}\mathfrak{Z}\in Z_{\RR}^{p}(\x,n)$
belongs to $\bigcap_{j,\epsilon}\ker(\imath_{j}^{*})$. Note that
we may have $\mathfrak{Z}\neq0$ but $\mathfrak{Z}^{\square}=0$.
In this case, $\Gamma_{\x,n}$ yields a map from $\left(\square_{\x}^{n},\partial\square_{\x}^{n}\right)\to\left(U^{\Delta},\partial U^{\Delta}\right)$,
which produces a splitting $\mathbb{E}^{\Delta}\to H^{2p-1}(\square_{\x}^{n},\partial\square_{\x}^{n})=H^{2p-n-1}(\x)$.
Hence $c_{\mathcal{B}}(\xi)=0=[\widetilde{AJ}_{\x}^{p}(0)]=[\widetilde{AJ}_{\Delta,\x}^{p}(\mathfrak{Z})],$
finishing the proof in this case.

So assume that $\mathfrak{Z}^{\square}$ is nonzero. Writing 
\[
(U^{\square},\partial U^{\square}):=\left(\square_{\x}^{n}\backslash|\mathfrak{Z}^{\square}|,\partial\square^{n}\backslash\{|\mathfrak{Z}^{\square}|\cap\partial\square_{\x}^{n}\}\right),
\]
we get an extension \small \begin{equation}\label{extension-sq}\xymatrix{H^{2p-n-1}(\x) \ar @{=} [d] \ar @{^(->} [r] & \mathbb{E}^{\square} \ar @{^(->} [d] \ar @{->>} [r] & \QQ(-p)= \langle \mathfrak{Z}^{\square} \rangle \ar @{^(->} [d]
\\
H^{2p-1}\left( \square^n_{\x},\partial \square^n_{\x} \right) \ar @{^(->} [r] & H^{2p-1}(U^{\square},\partial U^{\square}) \ar @{->>} [r] & H^{2p}_{|\mathfrak{Z}|}\left( \square^n_{\x},\partial \square^n_{\x} \right)^{\circ} }
\end{equation}\normalsize analogous to \eqref{extension}. In fact, $\Gamma_{\x,n}$
restricts to a map from $U^{\square}\to U^{\Delta}$ sending $\partial U^{\square}\to\partial U^{\Delta}$,
hence induces a map from the bottom row of \eqref{extension} to the
bottom row of \eqref{extension-sq}. By the end of Step 1, this is
an isomorphism on the left-hand terms. Since $\mathfrak{Z}^{\square}=\Gamma_{\x,n}^{*}\mathfrak{Z}$,
it also sends the $\QQ(0)$ to the $\QQ(0)$ and so gives an isomorphism
of the top rows.

Hence $c_{\mathcal{B}}(\xi)$ is the extension class \emph{also} of
the top row of \eqref{extension-sq}, which by \cite[Thm. 7.1]{KLM}
is computed by $AJ_{\x}^{p,n}(\mathfrak{Z}^{\square})$. Since $AJ_{\Delta,\x}^{p,n}(\mathfrak{Z})=AJ_{\x}^{p,n}(\mathfrak{Z}^{\square})$,
we are done.
\begin{rem}
A (much longer) direct proof of Theorem 3.2 could also be given, basically
by repeating the argument in $\S$5.8 and $\S$7 of \cite{KLM} in
the simplicial setting.
\end{rem}

\subsection*{Appendix I to $\S3$: proof of moving lemmas \ref{ML1}, \ref{ML2}}

We preface the actual proof with some simplicial algebra. Recall the
face maps $\rho_{i}:\Delta^{n-1}\hookrightarrow\Delta^{n}$ and define
degeneracy maps $\sigma_{i}:\Delta^{n+1}\to\Delta^{n}$ by
\[
[X_{0}:\cdots:X_{n+1}]\mapsto[X_{0}:\cdots:X_{i-1}:X_{i}+X_{i+1}:X_{i+2}:\cdots:X_{n+1}].
\]
For all $i=0,\ldots,n,$ set
\[
\partial_{i}:=(id_{\x}\times\rho_{i})^{*}:\, Z_{\Delta}^{p}(\x,n)\to Z_{\Delta}^{p}(\x,n-1)
\]
(so that $\partial=\sum_{i=0}^{n}(-1)^{i}\partial_{i}$) and
\[
s_{i}:=(id_{\x}\times\sigma_{i})^{*}:\, Z_{\Delta}^{p}(\x,n)\to Z_{\Delta}^{p}(\x,n+1).
\]
One has the relations\begin{equation}\label{ID} \begin{matrix} \del_i\del_j = \del_{j-1}\del_i&{\rm if}& i<j\\ \del_is_j = s_{j-1}\del_i&{\rm if}& i<j\\ \del_is_j = {\rm Id}&{\rm if}&i=j \ {\rm or}\ i=j+1\\ \del_is_j = s_j\del_{i-1}&{\rm if}&i>j+1\\ s_is_j = s_{j+1}s_i&{\rm if}&i\leq j \end{matrix} \end{equation}Also
recall the normalized complex with terms 
\[
N_{\Delta}^{p}(\x,n):=\cap_{i=0}^{n-1}\ker(\partial_{i})\subset Z_{\Delta}^{p}(\x,n).
\]

We introduce a filtration:
\[
\z^{p}(\x,\bullet)\supset\F^{0}\z^{p}(\x,\bullet)\supset\F^{1}\z^{p}(\x,\bullet)\supset\cdots\supset\n^{p}(\x,\bullet)
\]
as follows: for $\ell\geq0$, put
\[
\F^{\ell}\z^{p}(\x,n)=\big\{\xi\in\z^{p}(\x,n)\ \big|\ \del_{i}\xi=0,\ \forall\ 0\leq i<{\rm min}(n,\ell)\big\}.
\]
Let
\[
\lambda_{\ell}:\F^{\ell+1}\z^{p}(\x,\bullet)\subset\F^{\ell}\z^{p}(\x,\bullet)
\]
be the inclusion of chain complexes.
\begin{lem}
$\lambda_{\ell}$ is a quasi-isomorphism.\label{lemma 5.1}\end{lem}
\begin{proof}
Introduce\begin{equation}\label{EM} \kappa_{\ell} : \F^{\ell}\z^p(\x,\bullet)\to \F^{\ell+1}\z^p(\x,\bullet), \end{equation}by
the formula\begin{equation}\label{E665} \kappa_{\ell} (\xi) = \begin{cases} \xi&\text{if $\ell > n$}\\ \xi - s_{\ell}\del_{\ell}(\xi)&\text{if $\ell \leq n$} \end{cases} . \end{equation}We
claim that \eqref{EM} is a morphism of complexes.

To see this, first observe that $\kappa_{\ell}$ is the identity for
$\ell>n$, so it suffices to assume that $\ell\leq n$. Let $\xi\in\F^{\ell}\z^{p}(\x,n)$.
We must show that%
\footnote{Obviously both sides are zero if $\ell>n$.%
}
\[
\kappa_{\ell}\partial\xi=\partial\kappa_{\ell}\xi,
\]
i.e. that\begin{equation}\label{ID1}
\sum_{j=\ell}^{m}(-1)^j\big(\del_j\xi - s_{\ell}\del_{\ell}\del_j\xi\big) = \sum_{j=\ell+1}^{m}(-1)^j\big( \del_j\xi - \del_js_{\ell}\del_{\ell}\xi\big) . 
\end{equation}For $j\geq\ell+2$, we have
\[
s_{\ell}\del_{\ell}\del_{j}=s_{\ell}\del_{j-1}\del_{\ell}=\del_{j}s_{\ell}\del_{\ell}
\]
from \eqref{ID}. Thus with regard to \eqref{ID1}, we are reduced
to showing that\begin{multline*}
(-1)^{\ell}\big[\del_{\ell}\xi - s_{\ell}\del_{\ell}\del_{\ell}\xi\big] + (-1)^{\ell+1}\big[\del_{\ell+1}\xi - s_{\ell}\del_{\ell}\del_{\ell+1}\xi\big] =
\\ (-1)^{\ell+1}\big[\del_{\ell+1}\xi - \del_{\ell+1}s_{\ell}\del_{\ell}\xi\big].
\end{multline*}Using $\partial_{\ell+1}s_{\ell}=\text{Id}$, this is reduced to the
equation $\del_{\ell}\del_{\ell}-\del_{\ell}\del_{\ell+1}=0$, which
follows from \eqref{ID}. The claim is established.

Next observe that $\kappa_{\ell}\circ\lambda_{\ell}$ is the identity
on $\F^{\ell+1}\z^{p}(\x,\bullet)$. For $\xi\in\F^{\ell}\z^{p}(\x,n)$
we introduce the homotopy operator $T_{\ell}:\F^{\ell}\z^{p}(\x,n)\to\F^{\ell}\z^{p}(\x,n+1)$
by the formula\begin{equation*}
T_{\ell}(\xi) = \begin{cases} 0&\text{if $\ell > n$}\\ (-1)^{\ell}s_{\ell}(\xi)&\text{if $\ell \leq n$} \end{cases}.
\end{equation*}We will check that \begin{equation}\label{E666} \del T_{\ell}(\xi) + T_{\ell}\del(\xi) = \xi - (\lambda_{\ell}\circ \kappa_{\ell})(\xi), \end{equation}which
obviously implies that $\lambda_{\ell}\circ\kappa_{\ell}$ is homotopic
to the identity on $\F^{\ell}\z^{p}(\x,n)$. Firstly, from \eqref{E665},
the right-hand side of \eqref{E666} is given by \begin{equation}\label{E667} \xi - (\lambda_{\ell}\circ \kappa_{\ell})(\xi) =  s_{\ell}\del_{\ell}\xi . \end{equation}(Both
sides of \eqref{E667} are zero if $\ell>n$.) Next, the left-hand
side of \eqref{E666} is \begin{equation}\label{E668} (-1)^{\ell}[\del s_{\ell} + s_{\ell}\del](\xi). \end{equation}But
as $s_{\ell}:\z^{p}(\x,n)\to\z^{p}(\x,n+1)$, for $\ell\leq n$, using
\eqref{ID} (and $\xi\in\F^{\ell}$) gives \begin{multline*}
(-1)^{\ell}\del s_{\ell}(\xi) = (-1)^{\ell}\sum_{j=0}^{n+1}(-1)^j\del_js_{\ell}(\xi) =
\\  (-1)^{\ell}\sum_{j=0}^{\ell-1}(-1)^j\del_js_{\ell}(\xi)  +  [ \del_{\ell}-\del_{\ell+1}] s_{\ell}(\xi)  + (-1)^{\ell}\sum_{j=\ell+2}^{n+1}(-1)^j\del_js_{\ell}(\xi) =
\\
(-1)^{\ell}\sum_{j=\ell+2}^{n+1}(-1)^js_{\ell}\del_{j-1}(\xi) =
(-1)^{\ell}\sum_{j=\ell+1}^{n}(-1)^{j-1}s_{\ell}\del_j(\xi).
\end{multline*}Next, and again using $\xi\in\F^{\ell}$,
\[
(-1)^{\ell}s_{\ell}\del(\xi)=s_{\ell}\del_{\ell}(\xi)+(-1)^{\ell}\sum_{j=\ell+1}^{n}(-1)^{j}s_{\ell}\del_{j}(\xi).
\]
Then \eqref{E668} becomes $s_{\ell}\del_{\ell}(\xi)$, as required.
\end{proof}
Lemma \ref{lemma 5.1} has the following corollary. Let $\kappa:\z^{p}(\x,\bullet)\to\n^{p}(\x,\bullet)$
be the map of complexes defined by letting $\kappa^{(n)}:\z^{p}(\x,n)\to\n^{p}(\x,n)$
be the composite $\kappa_{n-1}\circ\kappa_{n-2}\circ\cdots\circ\kappa_{0}$.
Then the inclusion $\lambda:\n^{p}(\x,\bullet)\subset\z^{p}(\x,\bullet)$
induces an isomorphism on homology with inverse induced by $\kappa$;
moreover, $\kappa\circ\lambda$ is the identity on $\n^{p}(\x,\bullet)$.
So we get $\z^{p}(\x,\bullet)\cong\n^{p}(\x,\bullet)\oplus\ker\kappa$,
where
\[
\ker\kappa=D_{\Delta}^{p}(\x,\bullet):=\sum_{i=0}^{n-1}s_{i}\left(\z^{p}(\x,n-1)\right),
\]
and $D_{\Delta}^{p}(\x,\bullet)$ is acyclic.

Turning to the proofs of our moving lemmas, we consider the commutative
diagram\[ \xymatrix{
N^p_{\Delta,\RR}(\x,\bullet) \ar [r]^{i_1} \ar [d]_{i_2} & Z^p_{\Delta,\RR}(\x,\bullet) \ar [d]^{i_3} 
\\ 
\n^p(\x,\bullet) \ar [r]^{i_4} & \z^p(\x,\bullet) , } \]where $N_{\Delta,\RR}^{p}(\x,\bullet):=Z_{\Delta,\RR}^{p}(\x,\bullet)\cap\n^{p}(\x,\bullet).$
We have seen that $i_{4}$ is a quasi-isomorphism.

We claim that $i_{1}$ is a quasi-isomorphism. To see this, let
\[
\tau=\tau_{k,J}^{n+1}:=T_{\frac{X_{1}+\cdots+X_{n+1}}{-X_{0}}}\cap\cdots\cap T_{\frac{X_{k+1}+\cdots+X_{n+1}}{-X_{k}}}\cap\rho_{J}(\Delta^{n+1-|J|})
\]
be one of the real chains in $\Delta^{n+1}$. Then one checks that
$\sigma_{i}(\tau)\subset\Delta^{n}$ is contained in a $\tau_{k',J'}^{n}$
of the same real dimension (as $\sigma_{i}(\tau)$, not $\tau$). Reasoning
as in Step 2 of the Proof of Theorem \ref{AJ Thm}, we have that $s_{\ell}$
restricts to a map $Z_{\Delta,\RR}^{p}(\x,n)\to Z_{\Delta,\RR}^{p}(\x,n+1).$
By \eqref{E665}, it follows that $\kappa_{\ell}$ and $T_{\ell}$
also preserve ``subscript $\RR$'', so that the proof of Lemma \ref{lemma 5.1}
goes through with the real-intersection conditions, proving the claim.

It remains to show that $i_{2}$ is a quasi-isomorphism. The argument
in \cite[Appendix to 8.2]{KL} (cf. part (a)) proves exactly the same
thing in the cubical context. Replacing cubes with simplices and $\mathcal{T}^{n}$
by the iterated double $\mathcal{T}_{\Delta}^{n}:=D\left(\Delta_{\x}^{n};\rho_{0}(\Delta_{\x}^{n-1}),\ldots,\rho_{n}(\Delta_{\x}^{n-1})\right)$,%
\footnote{This is a singular variety (resembling the union of facets of a polytope)
with irreducible components all isomorphic to $\Delta_{\x}^{n}$,
and indexed by subsets of $\{0,\ldots,n\}$.%
} the same proof (using ideas of Levine \cite{Lv}) goes through \emph{mutatis
mutandis}. To give a flavor of the proof, we summarize the steps for
showing $i_{2}$ is ``quasi-surjective''. The idea is that any normalized
cycle $\mathfrak{Z}\in\ker(\del)\subset\n^{p}(\x,n)$ can, up to $\del\n^{p}(\x,n+1)$,
be described as the alternating pullback of a cycle on $\mathcal{T}_{\Delta}^{n}$.
This cycle in turn may be obtained by intersecting with a cycle $\mathcal{W}$
on a homogeneous space for $GL_{n}(K)$, where $K$ is the field of
definition of $\x$. Applying $g^{*}$ ($g\in GL_{n}(L)$, $L\supset K$)
to $\mathcal{W}$ and pulling back to $\x$ yields a cycle $\mathfrak{Z}'\in\n^{p}(\x_{L},n)$
which still only differs from $\mathfrak{Z}$ by an element of $\del\n^{p}(\x_{L},n+1)$.
By a variant of Kleiman transversality (cf. \cite{Lv}), one may choose
$g$ so that $\mathfrak{Z}'\in N_{\Delta,\RR}^{p}(\x_{L},n)$; a norm
argument then produces $\mathfrak{Z}''\in N_{\Delta,\RR}^{p}(\x,n)$
in the same class as $\mathfrak{Z}$.

\subsection*{Appendix II to $\S3$: verification that $R_{\mathfrak{Z}}^{\Delta}\in D^{2p-n-1}(\x)$}

We consider progressively more general cases, with $\mathfrak{Z}\subset\x\times\PP^{n}$
always irreducible and giving an element of $Z_{\Delta,\RR}^{p}(\x,n)$:

$\underline{\text{Case 1}}$: \emph{$p=n$, with $\overline{\pi_{\x}(\mathfrak{Z})}=\x$
and $\mathfrak{Z}$ generically of degree $1$ over $\x$.} Writing
\[
\underline{f}=(f_{1},\ldots,f_{n}):=\mathfrak{Z}^{*}\left(-\frac{X_{1}+\cdots+X_{n}}{X_{0}},\ldots,-\frac{X_{n}}{X_{n-1}}\right),
\]
we define subvarieties 
\[
H_{\underline{f}}:=\left|\mathfrak{Z}^{*}\left((X_{0}+\cdots+X_{n})\right)\right|,
\]
 
\[
\y_{\underline{f}}:=\bigcup_{j=1}^{n}\left|(1-f_{j})_{0}\right|,\text{ and }D_{\underline{f}}:=\bigcup_{j=1}^{n}\left|(f_{j})\right|
\]
of $\x$. Let $\omega\in A^{2\dim(\x)-n+1}(\x)$ be a $C^{\infty}$
test form; we must show that \begin{equation}\label{eqn converge}\int_{\x} R(\underline{F})\wedge \omega := \lim_{\epsilon\to 0}\int_{\x\setminus \mathcal{N}_{\epsilon}(D_{\underline{f}})} R(\underline{f})\wedge \omega
\end{equation} is finite (where $\mathcal{N}_{\epsilon}(\cdot)$ denotes a small
tubular neighborhood). Write $\mathcal{E}_{\underline{f},\omega}$
for the union of irreducible components $W$ of $D_{\underline{f}}$
along which every term of $R(\underline{f})\wedge\omega$ has a factor
of $dw$, $d\bar{w}$, $w$, or $\bar{w}$, where $w$ is an algebraic
(and locally holomorphic) function with $W$ in its zero-set. More
precisely, if $J_{W}:=\left\{ j\in\{1,\ldots,n\}\left||(f_{j})|\supset W\right.\right\} =\{j_{1},\ldots,j_{k}\},$
then $R_{\underline{f}}\wedge\omega$ breaks into terms $\log f_{j_{\ell}}\text{dlog}f_{j_{\ell+1}}\wedge\cdots\wedge\text{dlog}f_{j_{k}}\cdot\delta_{T_{f_{j_{1}}}\cap\cdots\cap T_{f_{j_{\ell-1}}}}$
with $\alpha$ a monomial $C^{\infty}$ $(\dim(\x)-k+1)$-form in
coordinates $\{z_{1}=w,\ldots,z_{n}\}$, and we require that $\alpha$
contain a $w$, $\bar{w}$, $dw$, or $d\bar{w}$.

First assume that $D_{\underline{f}}$ is a normal crossing divisor.
In that event, it will suffice to bound \eqref{eqn converge} in a
neighborhood of a general point of each irreducible component of $D_{\underline{f}}$,
since the bounds near intersection (higher codimension) points will
break into products of codimension-$1$ bounds. The only possibilities
for nonconvergence along $W$ are terms of the form \begin{equation}\label{eqn nonconverge}\int_{T_{w^a}}\log(w^b) \text{dlog}(w^c) \wedge C^{\infty} \;\;\text{and} \;\; \int_{T_{w^a}}\text{dlog}(w^b)\wedge C^{\infty}.
\end{equation} where without loss of generality one can take the integers $a,b,c$
to be $1$. Evidently the presence of a $dw$, $d\bar{w}$, $w$,
or $\bar{w}$ in each monomial term of the $C^{\infty}$ expression
makes \eqref{eqn nonconverge} converge, so that we only need to worry
about $W\nsubseteq\mathcal{E}_{\underline{f},\omega}$. But $H_{\underline{f}}\subset\y_{\underline{f}}\subset\mathcal{E}_{\underline{f},\omega}$,
and $\y_{\underline{f}}$ also contains every $W$ along which the
numerator of $f_{1},\, f_{2},\,\ldots,$ or $f_{n-1}$ vanishes, while
outside $H_{\underline{f}}$ only one of $X_{0},\ldots,X_{n}$ can
vanish in codimension one. Consequently, each component of $D_{\underline{f}}\backslash\mathcal{E}_{\underline{f},\omega}$
can only be contained in one $|(f_{j})|$ and \eqref{eqn nonconverge}
cannot occur.

If $D_{\underline{f}}$ does not have normal crossings, consider an
embedded resolution\[ \xymatrix{\widetilde{\x} \ar @{->>} [r]^{\beta} & \x \\ 
\widetilde{D_{\underline{f}}}\cup E_{\beta} \ar @{^(->} [u] \ar @{->>} [r] & D_{\underline{f}} \ar @{^(->} [u]} \]where $\widetilde{D_{\underline{f}}}$ is the proper transform and
$E_{\beta}$ the exceptional divisor (with union a NCD). By a simple
computation, $\mathcal{E}_{\beta^{*}\underline{f},\beta^{*}\omega}\supset\beta^{-1}(\mathcal{E}_{\underline{f},\omega})$
and we only need to consider components $W\underset{\text{loc}}{=}\{w'=0\}$
of $E_{\beta}$ in the preimage of $\x\backslash\y_{\underline{f}}$.
But then by the proper intersection conditions on $\mathfrak{Z}$,
$|J_{W'}|$ is bounded by $c:=\text{codim}_{\x}(\beta(W'))$. In particular,
if $w_{1}=\cdots=w_{c}=0$ locally cuts out $\beta(W')$, we have
in each term of $R(\underline{f})\wedge\omega$ a $dw_{i}$, $d\bar{w}_{i}$,
$w_{i}$, or $\bar{w}_{i}$ factor ($i\in\{1,\ldots,c\}$), hence
in each term of $R(\beta^{*}\underline{f})\wedge\beta^{*}\omega$
a $dw'$, $d\bar{w}'$, $w'$, or $\bar{w}'$ factor. Conclude that
$W'$ hence $E_{\beta}$ is contained in $\mathcal{E}_{\beta^{*}\underline{f},\beta^{*}\omega}$,
proving convergence.

$\underline{\text{Case 2}}$: \emph{Remove the degree-$1$ assumption
(so $\mathfrak{Z}$ is simply finite over $\x$).} The above argument
goes through for the branches of $\mathfrak{Z}$, when one considers
that the expressions in \eqref{eqn nonconverge} are not essentially
different if we take $a,b,c\in\QQ$, and that codimension in $\mathfrak{Z}$
is codimension in $\x$.

$\underline{\text{Case 3}}$:\emph{ $p>n$ and $\mathfrak{Z}$ generically
finite over a subvariety $V$ of $\x$}. At first glance, one has
to worry about the failure of proper intersection conditions for the
base-change of $\mathfrak{Z}$ under a desingularization $\tilde{V}\twoheadrightarrow V$.
(Otherwise, we are reduced to Case 2.) But as in the end of Case 1,
away from the sets $\mathfrak{Z}\cap(V\times\{X_{j}+\cdots+X_{n}=0\})$
($j=0,\ldots,n-1$), the number of singular $\delta_{T}$ or $\text{dlog}$
factors is bounded by the codimension of the corresponding subvariety
of $\mathfrak{Z}$ (hence $V$), and then a similar argument holds.

$\underline{\text{Case 4}}$: \emph{general case.} Working locally,
there is a finite projection of $\mathfrak{Z}$ to $V\times\PP^{k}$
for some $k<n$, and we are done by Case 3.

\section{Milnor reciprocity laws}

The telescoping property (Prop. \ref{prop telescope}) of the simplicial
currents $R_{n}^{\Delta},S_{n}^{\Delta}$ makes them particularly
suitable for the study of reciprocity laws arising from subvarieties
of projective space. We shall begin, however, from a more general
and ``intrinsic'' perspective, which is independent of the choice
of simplicial vs. cubical.

Let $\x$ be a smooth complete curve over $\CC$, and $f,g\in\CC(\x)^{*}$.
Writing \begin{equation}\label{eqn Tame}\left\{ \begin{array}{c}{\text{Tame}}_{p}:\, K_{2}^{M}(\CC(\x))\to K_{1}^{M}(\CC)\cong\CC^{*}\\\{f,g\}\mapsto\lim_{x\to p}(-1)^{\nu_{p}(f)\nu_{p}(g)}\frac{f(x)^{\nu_{p}(g)}}{g(x)^{\nu_{p}(f)}}\end{array}\right.
\end{equation}for $p\in\x(\CC)$, \emph{Weil reciprocity} states that the (finite)
product
\[
\prod_{p\in\x(\CC)}\text{Tame}_{p}\{f,g\}=1.
\]
This result gives rise to several other reciprocity laws in higher
dimension. For example, Parshin {[}resp. bilocal{]} reciprocity (cf.
\cite{Ho}) on an algebraic surface $\mathcal{S}$ is obtained by
\emph{applying} Weil recirocity on a curve $\x\subset\mathcal{S}$
to $\text{Tame}_{\x}\xi$ {[}resp. $\{\text{Tame}_{\x}\mu,\text{Tame}_{\x}\eta\}${]}
for $\xi\in K_{3}^{M}(\CC(\mathcal{S}))$ {[}resp. $\mu\otimes\eta\in K_{2}^{M}(\CC(\mathcal{S}))^{\otimes2}${]}.
Suslin reciprocity (cf. \cite{Ke3}) generalizes Weil to higher $K$-theory,
replacing \eqref{eqn Tame} by $\text{Tame}_{p}:\, K_{3}^{M}(\CC(X))\to K_{2}^{M}(\CC)$.

The generalizations we pursue here take a different direction, and
begin from the
\begin{prop}
Let $D=\{p_{1},\ldots,p_{r}\}\subset\x$ and $\x^{*}:=\x\backslash D$;
then for each $p\geq2$, the composition 
\[
CH^{p}(\x^{*},2p-2)\overset{\oplus_{\alpha}Res_{p_{\alpha}}}{\longrightarrow}\oplus_{\alpha}CH^{p-1}(\CC,2p-3)\overset{AJ}{\longrightarrow}\oplus_{\alpha}\CC/\ZZ(p-1)
\]
has image in the kernel of the augmentation map $\oplus_{\alpha}\CC/\ZZ(p-1)\overset{\sum}{\to}\CC/\ZZ(p-1).$
\end{prop}
This is easily proved from the localization sequence and its compatibility
with the $AJ$ map, or using Reciprocity Law A below. The case $p=2$
is Weil reciprocity, while $p=3$ {[}resp. $4,\,\ldots\,${]} is related
to the dilogarithm {[}resp. trilogarithm, $\ldots\,${]} at algebraic
arguments and more generally special values of $L$-functions. So
for polylogarithmic functional equations with \emph{variable }arguments,
this is not the way to go.

At the next stage of generalization, where $\x/\CC$ is any smooth
projective variety, we encounter an unpleasant reality when $\dim\x=:d>1$.
Consider a codimension-one subvariety $D\subset\x$ with irreducible
components $\{D_{\alpha}\}$ and smooth locus $\cup D_{\alpha}^{*}$,
and write $\x^{*}:=\x\backslash D$. Taking $p>d$, for any $\alpha$
the composition\small  
\[
CH^{p}(\x^{*},2(p-d))\overset{Res_{\alpha}}{\longrightarrow}CH^{p-1}(D_{\alpha}^{*},2(p-d)-1)\overset{AJ}{\longrightarrow}H^{2(d-1)}(D_{\alpha}^{*},\CC/\ZZ(p-1))
\]
\normalsize is zero unless $D_{\alpha}^{*}=D_{\alpha}$, so that
integrating the image current does not give a well-defined number
in $\CC/\ZZ(p-1)$. So we are forced to work on the level of precycles,
which yields
\begin{prop}
\label{prop 4.2}Let $p>d\geq1$, $n:=2(p-d)$, and%
\footnote{The parentheses $(\Delta)$ mean that we may work in either the simplicial
or the cubical setting.%
} $\mathfrak{Z}\in Z_{\RR(,\Delta)}^{p}(\x,n)$ be a precycle with
$\partial_{B}\mathfrak{Z}$ supported on $D$. Then writing $\partial_{B}\mathfrak{Z}=:\sum\imath_{*}^{D_{\alpha}}Res_{\alpha}\mathfrak{Z}$
\emph{(}with $Res_{\alpha}\mathfrak{Z}\in Z_{\RR(,\Delta)}^{p-1}(D_{\alpha},n-1)$\emph{)},
we have
\[
\sum_{\alpha}\int_{D_{\alpha}}R_{Res_{\alpha}\mathfrak{Z}}^{(\Delta)}\equiv0\;\;\;\text{mod }\ZZ(n-1).
\]
\end{prop}
\begin{proof}
Note that $R_{Res_{\alpha}\mathfrak{Z}}^{(\Delta)}$ is a current
of top degree $2(p-1)-(n-1)-1=2(d-1)$ on $D_{\alpha}$. Since $p>d$,
$F^{p}D^{2d}(\x)=\{0\}$ and $\Omega_{\mathfrak{Z}}^{(\Delta)}=0$.
So \eqref{eqn dR} becomes
\[
d[R_{\mathfrak{Z}}^{(\Delta)}]=-(2\pi i)^{n}\delta_{T_{\mathfrak{Z}}^{(\Delta)}}+2\pi i\sum_{\alpha}\imath_{*}^{D_{\alpha}}R_{Res_{\alpha}\mathfrak{Z}}^{(\Delta)},
\]
from which the result follows by Stokes's theorem.
\end{proof}
Restricting to the case $n=p=2d$, suppose $F_{0},\ldots,F_{n}\in\Gamma(\x,\mathcal{O}_{\x}(k))$
is an $n$-tuple of homogeneous functions such that 
\[
\Gamma_{\underline{F}}:=\left\{ \left.\left(x,[F_{0}(x):\cdots:F_{n}(x)]\right)\right|x\in\x(\CC)\right\} \in Z_{\RR,\Delta}^{n}(\x,n).
\]
Writing $\sum m_{ij}D_{ij}:=(F_{i})$, one obtains 
\[
\partial_{B}\Gamma_{\underline{F}}^{\Delta}=\sum_{i=0}^{n}(-1)^{i}\sum_{j}m_{ij}\imath_{*}^{D_{ij}}\Gamma_{[F_{0}:\cdots:\widehat{F_{i}}:\cdots:F_{n}]}^{\Delta},
\]
which together with Proposition \ref{prop 4.2} gives the 
\begin{cor}
We have 
\[
\sum_{i=0}^{n}(-1)^{i}\sum_{j}m_{ij}\int_{D_{ij}^{*}}R^{\Delta}(X_{0}:\cdots:\widehat{X_{i}}:\cdots:X_{n})\underset{\ZZ(n-1)}{\equiv}0.
\]

\end{cor}
We leave to the reader the obvious analogue for the cubical Milnor
regulator currents $R(f_{1},\ldots,\widehat{f_{i}},\ldots,f_{n})$.
Note that the $n=2$ case of this is Weil reciprocity for functions
$f_{1},f_{2}\in\CC(\x)^{*}$ with $|(f_{1})|\cap|(f_{2})|=\emptyset$.

The Corollary has a natural ``extrinsic'' analogue for algebraic
cycles in even-dimensional projective space. We lose no generality
by stating this result, which is our first main point, for subvarieties. 
\begin{defn}
\label{defn general pos}We shall say that a subvariety of $\PP^{M}$
is in \emph{general position} if it properly intersects all chains
of the form $\imath_{*}^{J}(T_{-\frac{X_{1}}{X_{0}}}\cap\cdots\cap T_{-\frac{X_{k}}{X_{k-1}}})$
where $\imath^{J}:\PP^{M-|J|}\hookrightarrow\PP^{M}$ sends $[Z_{0}:\cdots:Z_{M-|J|}]$
to the projective $(M+1)$-tuple obtained by inserting zeroes at the
positions $j_{1},\ldots,j_{|J|}$.\end{defn}
\begin{thm}
\textbf{\emph{\label{thm rec law A}(Reciprocity Law A)}} Let $\mathfrak{R}_{m}$
stand for $R_{m}^{\Delta}$ or $S_{m}^{\Delta}$, and $\x\subset\PP^{2d}$
be an irreducible subvariety of dimension $d$, with $\y_{i}:=\x\cdot(X_{i})$
for $i=0,\ldots,2d$. Assuming that $\x$ is in general position,
we have
\[
\sum_{j=0}^{2d}(-1)^{j}\int_{\y_{_{j}}^{*}}\frac{1}{(2\pi i)^{d-1}}\mathfrak{R}_{2d-1}\left(X_{0}:\cdots:\widehat{X_{j}}:\cdots:X_{2d}\right)\underset{\ZZ(d)}{\equiv}0.
\]
\end{thm}
\begin{proof}
The general position assumption allows us to pull back the result
of Proposition \ref{prop telescope}. Noting that by Hodge type we
have $\imath_{\x}^{*}\Omega_{2d}^{\Delta}=0$, this gives
\[
d[\imath_{\x}^{*}\mathfrak{R}_{2d}]+(2\pi i)^{2d}\delta_{\x\cdot T_{2d}^{\Delta}}=2\pi i\sum_{j=0}^{2d}(-1)^{j}\imath_{\x}^{*}(\rho_{j})_{*}\mathfrak{R}_{2d-1}.
\]
Dividing by $(2\pi i)^{d}$ and integrating over $\x$ gives the result.
\end{proof}
We have written it in this form because the first term of (say) $S_{2d-1}(X_{1}:\cdots:X_{2d})$
whose pullback to $\y_{0}^{*}$ does not vanish, is
\[
(2\pi i)^{d-1}\delta_{T_{-\frac{X_{2}}{X_{1}}}\cap\cdots\cap T_{-\frac{X_{d}}{X_{d-1}}}}\log\left(\frac{-X_{d+1}}{X_{d}}\right)\text{dlog}\left(\frac{X_{d+2}}{X_{d+1}}\right)\wedge\text{dlog}\left(\frac{X_{2d}}{X_{2d-1}}\right).
\]
For $\x\cong\PP^{d}$ a linear subvariety, one expects Theorem \ref{thm rec law A}
to translate into functional equations for (a variant of) $Li_{d}$.
It turns out that the $S_{m}^{\Delta}$ version of the result, which
allows for more singular integrals, is much more suited to making
this connection.

There is also a natural ``projective dual'' to Theorem \ref{thm rec law A},
which we shall only state for the $S_{m}^{\Delta}$. (We do not know
if an analogue of Lemma \ref{lemma 2 on Xi} holds for the $R_{m}^{\Delta}$.)
In first approximation, one would expect a statement of the following
form: given $\x\subset\PP^{2d}$ general of dimension $d-1$, the
alternating sum
\[
\sum_{j=0}^{2d}(-1)^{j}\int_{\x}\frac{1}{(2\pi i)^{d-1}}S_{2d-1}^{\Delta}\left(X_{0}:\cdots:\widehat{X_{j}}:\cdots:X_{2d}\right)
\]
is zero mod $\ZZ(d)$. (Note that this morally involves projecting
$\x$ to the coordinate hyperplanes in $\PP^{2d}$, rather than intersecting
with them.) This turns out to require correction terms, essentially
because \emph{complex}-valued regulator currents cannot be made exactly
alternating multilinear in their arguments. 

In order to make the corrections, we shall require two lemmas. Introduce
the notation
\[
S_{alt}^{k}:=\sum_{j=0}^{k+2}(-1)^{j}S_{k+1}^{\Delta}\left(X_{0}:\cdots:\widehat{X_{j}}:\cdots:X_{k+2}\right)\in D^{k}(\PP^{k+2}),
\]
\[
I_{*}^{k+2}:=\sum_{j=0}^{k+2}(-1)^{j}\rho_{*}^{j}:\, D^{*-2}(\PP^{k+1})\to D^{*}(\PP^{k+2}),
\]
where we recall $\rho^{j}:\PP^{k+1}\hookrightarrow\PP^{k+2}$ is the
inclusion of the $j^{\text{th}}$ coordinate hyperplane. Note that
$I_{*}^{\ell+1}\circ I_{*}^{\ell}=0$. For $k$ odd, let $P_{k}$
denote a fixed $\PP^{\frac{k+5}{2}}\subset\PP^{k+2}$. To motivate
the first lemma, observe that on $\PP^{2}$
\[
S_{alt}^{0}=\log\left(-\frac{X_{2}}{X_{1}}\right)-\log\left(-\frac{X_{2}}{X_{0}}\right)+\log\left(-\frac{X_{1}}{X_{0}}\right)
\]
\[
=:\pi i\delta_{\Gamma_{012}}
\]
takes values $\pm\pi i$, making $\Gamma_{012}$ an integral $4$-chain
(or $\ZZ$-valued $0$-current). A computation shows that $S_{alt}^{1}=$\small 
\[
R_{2}\left(-\frac{X_{2}}{X_{1}},-\frac{X_{3}}{X_{2}}\right)-R_{2}\left(-\frac{X_{2}}{X_{0}},-\frac{X_{3}}{X_{2}}\right)+R_{2}\left(-\frac{X_{1}}{X_{0}},-\frac{X_{3}}{X_{1}}\right)-R_{2}\left(-\frac{X_{1}}{X_{0}},-\frac{X_{2}}{X_{1}}\right)
\]
\normalsize 
\[
=d\left\{ \pi i\log\left(-\frac{X_{3}}{X_{2}}\right)\delta_{\Gamma_{012}}\right\} -\frac{1}{2}(2\pi i)^{2}\delta_{T_{-\frac{X_{1}}{X_{0}}}\cap\Gamma_{123}},
\]
which forms the base case for 
\begin{lem}
There exists a sequence of currents $\Xi^{k}\in D^{k}(\PP^{k+3})$
\emph{(}$k=0,1,2,\ldots$\emph{)} and constants $\alpha_{1},\alpha_{3},\alpha_{5},\ldots\in\CC$
such that for each $k\geq0$\begin{equation}\label{eqn S_alt/Xi}S_{alt}^{k+1}+2\pi iI_{*}^{k+3}\Xi^{k-1}\equiv\left\{ \begin{array}{cc}d\Xi^{k}, & k\text{ even}\\d\Xi^{k}+\alpha_{k}\delta_{I_{*}^{k+3}P_{k}}, & k\text{ odd}\end{array}\right.
\end{equation}modulo $\mathfrak{C}_{k+2}:=\frac{1}{2}\ZZ(k+2)$-valued chains.\end{lem}
\begin{proof}
By Proposition \ref{prop telescope} and the fact that 
\[
\sum_{j=0}^{n+3}(-1)^{j}\Omega_{n+2}^{\Delta}\left(X_{0}:\cdots:\widehat{X_{j}}:\cdots:X_{n+3}\right)=0
\]
on $\PP^{n+3}$, we have for each $n$
\[
dS_{alt}^{n+1}\equiv-2\pi iI_{*}^{n+3}S_{alt}^{n}\;\;\;(\text{mod }\mathfrak{C}_{n+2}).
\]
Inductively assuming \eqref{eqn S_alt/Xi} for $k=n-1$, this gives
\[
dS_{alt}^{n+1}\equiv-2\pi iI_{*}^{n+3}\{-2\pi iI_{*}^{n+2}\Xi^{n-2}+d\Xi^{n-1}\,\underset{\text{if }n\text{ even}}{\underbrace{[+\alpha_{n-1}\delta_{I_{*}^{n+2}P_{n-1}}]}}\}
\]
\[
\equiv-2\pi id\left\{ I_{*}^{n+3}\Xi^{n-1}\right\} 
\]
\[
\implies\; S_{alt}^{n+1}+2\pi iI_{*}^{n+3}\Xi^{n-1}\text{ is closed (mod }\mathfrak{C}_{n+2}\text{)}.
\]
If $n$ is even, we are done since $H^{n+1}(\PP^{n+3})=\{0\}.$ Otherwise,
noting that $[I_{*}^{n+3}P_{n}]=[\PP^{\frac{n+5}{2}}]\in H^{n+1}(\PP^{n+3}),$
there exist $\alpha\in\CC$ and $\Xi^{n}\in D^{n}(\PP^{n+3})$ such
that
\[
S_{alt}^{n+1}\equiv-2\pi iI_{*}^{n+3}\Xi^{n-1}+d\Xi^{n}+\alpha\delta_{I_{*}^{n+3}P_{n}}\;\text{(mod }\mathfrak{C}_{n+2}\text{).}
\]

\end{proof}
In fact, a more detailed computation reveals that with the right choices
of the $\{\Xi^{k}\}$, the $\{\alpha_{k}\}$ may be taken to be $0$:
\begin{lem}
\label{lemma 2 on Xi}One has for each $k\geq0$\begin{equation}\label{eqn *!}S_{alt}^{k+1}+2\pi iI_{*}^{k+3}\Xi^{k-1}\equiv d\Xi^{k}\;(\text{mod }\mathfrak{C}_{k+2}),
\end{equation}where 
\[
\Xi^{k}=\pi i\sum_{\ell=0}^{k}(-2\pi i)^{\ell}\delta_{\Gamma^{\ell}}\log\left(-\frac{X_{\ell+3}}{X_{\ell+2}}\right)\text{dlog}\left(\frac{X_{\ell+4}}{X_{\ell+3}}\right)\wedge\cdots\wedge\text{dlog}\left(\frac{X_{k+3}}{X_{k+2}}\right)
\]
and the codimension-$\ell$ chain%
\footnote{The widehats mean that those two $T$'s are omitted from the intersection.%
}
\[
\Gamma^{\ell}=\sum_{j=0}^{\ell}\Gamma_{j,j+1,j+2}T_{-\frac{X_{1}}{X_{0}}}\cap\cdots\cap\widehat{T_{-\frac{X_{j+1}}{X_{j}}}}\cap\widehat{T_{-\frac{X_{j+2}}{X_{j+1}}}}\cap\cdots\cap T_{-\frac{X_{\ell+2}}{X_{\ell+1}}}.
\]
\end{lem}
\begin{proof}
(Sketch) The main step is to show directly that $S_{alt}^{k+1}=$\small 
\[
\sum_{\ell=0}^{k+1}(-2\pi i)^{\ell}\left\{ \begin{array}{c}
\delta_{T_{alt}^{\ell}}\log\left(-\frac{X_{\ell+2}}{X_{\ell+1}}\right)+\\
(-1)^{\ell}\pi i\delta_{T^{\ell}\cap\Gamma_{\ell,\ell+1,\ell+2}}
\end{array}\right\} \text{dlog}\left(\frac{X_{\ell+3}}{X_{\ell+2}}\right)\wedge\cdots\wedge\text{dlog}\left(\frac{X_{k+3}}{X_{k+2}}\right),
\]
\normalsize where $T^{\ell}:=T^{\ell}[0:\cdots:\ell]=T_{-\frac{X_{1}}{X_{0}}}\cap\cdots\cap T_{-\frac{X_{\ell}}{X_{\ell-1}}}$
and 
\[
T_{alt}^{\ell}:=\sum_{j=0}^{\ell+1}(-1)^{j}T^{\ell}[0:\cdots:\widehat{j}:\cdots:\ell].
\]
(Note that the term in braces is just $\pi i\delta_{\Gamma_{012}}$
for $\ell=0$.) To verify \eqref{eqn *!}, one then uses the formula
$\frac{1}{2}\partial\Gamma^{\ell}=-T_{alt}^{\ell+1}+\{\text{boundary terms}\}$,
the first case of which is
\[
\frac{1}{2}\partial\Gamma^{0}=\frac{1}{2}\partial\Gamma_{012}=T_{-\frac{X_{2}}{X_{0}}}-T_{-\frac{X_{1}}{X_{0}}}-T_{-\frac{X_{2}}{X_{1}}}=-T_{alt}^{1},
\]
and $\Gamma^{\ell-1}\cap T_{-\frac{X_{\ell+2}}{X_{\ell+1}}}+T^{\ell}\cap\Gamma_{\ell,\ell+1,\ell+2}=\Gamma^{\ell}.$
Details are left to the reader.
\end{proof}
We can now state
\begin{thm}
\textbf{\emph{\label{thm Rec Law B}(Reciprocity Law B)}} Let $\x\subset\PP^{2d}$
be an irreducible subvariety of dimension $d-1$, with $\y_{i}:=\x\cdot(X_{i})$
for $i=0,\ldots,2d.$ Assuming that $\x$ and its projections to the
coordinate hyperplanes are in general position, we have
\[
0\equiv\sum_{j=0}^{2d}(-1)^{j}\int_{\x}\frac{1}{(2\pi i)^{d-1}}S_{2d-1}^{\Delta}\left(X_{0}:\cdots:\widehat{X_{j}}:\cdots:X_{2d}\right)\mspace{50mu}
\]
\[
\mspace{50mu}+\sum_{j=0}^{2d}(-1)^{j}\int_{\y_{j}}\frac{1}{(2\pi i)^{d-2}}\Xi^{2d-4}\left(X_{0}:\cdots:\widehat{X_{j}}:\cdots:X_{2d}\right)
\]
modulo $\frac{1}{2}\ZZ(d)$.\end{thm}
\begin{proof}
Follows immediately from Lemma \ref{lemma 2 on Xi} with $k=2d-3$.
\end{proof}
The correction terms $\int_{\y_{j}}\Xi^{2d-3}(\cdots)$, as we shall
see, may be thought of as ``lower-weight'' in the context of linear
subvarieties and polylogarithms. In essence, one is trading off the
formal simplicity of Reciprocity Law A for greater algebraic simplicity
in the arguments of the expected $Li_{d}$ terms $\int_{\x}S_{2d-1}^{\Delta}(\cdots)$.

\section{Functional equations for $Li_{2}$}

To illustrate the different strengths of the two reciprocity laws
of the last section, we shall apply both to obtain different forms
of the 5-term relation for the dilogarithm
\[
Li_{2}(z)=-\int_{0}\log(1-z)\frac{dz}{z}.
\]
Reciprocity Law A involves intersecting an $\x^{d}\subset\PP^{2d}$
with the coordinate hyperplanes. Taking $d=2$, let $\x$ be the $\PP^{2}\subset\PP^{4}$
obtained by projectivizing the row-space of
\[
\left(\begin{array}{ccccc}
1 & 1 & -1 & 0 & 0\\
\frac{1}{x} & 1 & 0 & -1 & 0\\
\frac{1}{y} & 1 & 0 & 0 & -1
\end{array}\right).
\]
The intersections $\y_{i}$ \label{eqn 5 matrices}($i=0,\ldots,4$)
are given by projectivizing the sub-row-spaces with $X_{i}=0$ (and
deleting the $i^{\text{th}}$ column):\begin{equation}\left\{ \begin{array}{cc}\left(\begin{array}{cccc}\frac{x-1}{x} & \frac{1}{x} & -1 & 0\\\frac{y-1}{y} & \frac{1}{y} & 0 & -1\end{array}\right) & i=0\\\left(\begin{array}{cccc}\frac{1-x}{x} & 1 & -1 & 0\\\frac{1-y}{y} & 1 & 0 & -1\end{array}\right) & i=1\\\left(\begin{array}{cccc}\frac{1}{x} & 1 & -1 & 0\\\frac{1}{y} & 1 & 0 & -1\end{array}\right) & i=2\\\left(\begin{array}{cccc}1 & 1 & -1 & 0\\\frac{1}{y} & 1 & 0 & -1\end{array}\right) & i=3\\\left(\begin{array}{cccc}1 & 1 & -1 & 0\\\frac{1}{x} & 1 & 0 & -1\end{array}\right) & i=4\end{array}\right.
\end{equation}Let $\y$ be the $\PP^{1}\subset\PP^{3}$ given by 
\[
\begin{array}{c}
1\\
-t
\end{array}\left(\begin{array}{cccc}
a & c & -1 & 0\\
b & d & 0 & -1
\end{array}\right),
\]
where the notation means that $t$ parametrizes $\y$ by $t\mapsto[a-bt:c-dt:-1:t].$
On $\PP^{3}$, we have $\frac{1}{2\pi i}S_{3}^{\Delta}(X_{0}:X_{1}:X_{2}:X_{3})=$
\[
\frac{1}{2\pi i}\log\left(-\frac{X_{1}}{X_{0}}\right)\text{dlog}\left(\frac{X_{2}}{X_{1}}\right)\wedge\text{dlog}\left(\frac{X_{3}}{X_{2}}\right)+\log\left(-\frac{X_{2}}{X_{1}}\right)\text{dlog}\left(\frac{X_{3}}{X_{2}}\right)\delta_{T_{-\frac{X_{1}}{X_{0}}}}
\]
\[
+2\pi i\log\left(-\frac{X_{3}}{X_{2}}\right)\delta_{T_{-\frac{X_{1}}{X_{0}}}\cap T_{-\frac{X_{2}}{X_{1}}}}.
\]
Only the middle term survives the pullback to $\y$, since $\text{dlog}\wedge\text{dlog}=0$
and $T_{-\frac{c-dt}{a-bt}}\cap T_{c-dt}$ is the closure of the intersection
of two open arcs that do not meet. So we must compute
\[
\frac{1}{2\pi i}\int_{\y}S_{3}^{\Delta}=-\int_{T_{-\frac{c-dt}{a-bt}}}\log(c-dt)\text{dlog}(t)
\]
\[
=-\int_{\frac{a}{b}}^{\frac{c}{d}}\left\{ \log c+\log\left(1-\frac{d}{c}t\right)\right\} \text{dlog}(t)
\]
\[
=Li_{2}(1)-Li_{2}\left(\frac{ad}{bc}\right)+\log(c)\log\left(\frac{ad}{bc}\right).
\]
Taking the alternating sum over the 5 matrices \eqref{eqn 5 matrices},
Theorem \ref{thm rec law A} gives the Abel-Spence relation%
\footnote{combine (1.22) (with $x\mapsto1-x$) and (1.11) (with $z=x$) in \cite{Le}%
}\begin{multline} \label{ASrel} 0 = Li_{2}(x)-Li_{2}(y)+Li_{2}\left( \frac{y}{x} \right) - Li_{2}\left( \frac{y(1-x)}{x(1-y)} \right) + Li_{2} \left( \frac{1-x}{1-y} \right) \\ -Li_{2}(1)+\log(x)\log\left(\frac{1-x}{1-y}\right) . \end{multline}

For a demonstration of Reciprocity Law B, we will need the integral
of $\frac{1}{2\pi i}S_{3}^{\Delta}$ over the most general form\begin{equation}\label{eqn general P^1}\begin{array}{c}t\\1\end{array}\left(\begin{array}{cccc}a_0 & a_1 & a_2 & a_3\\b_0 & b_1 & b_2 & b_3\end{array}\right)
\end{equation}of $\y\cong\PP^{1}\subset\PP^{3}$. Using the substitution $v=-\frac{a_{3}t+b_{3}}{a_{2}t+b_{2}}$
and denoting the minor $a_{i}b_{j}-a_{j}b_{i}$ by $|ij|$, this is
\[
\int_{T_{-\frac{a_{1}t+b_{1}}{a_{0}t+b_{0}}}}\log\left(-\frac{a_{2}t+b_{2}}{a_{1}t+b_{1}}\right)\text{dlog}\left(-\frac{a_{3}t+b_{3}}{a_{2}t+b_{2}}\right)=
\]
\[
\int_{-\frac{|03|}{|02|}}^{-\frac{|13|}{|12|}}\log\left(\frac{-|23|}{|12|v+|13|}\right)\text{dlog}(v)=
\]
\[
-\log\left(-\frac{|23|}{|13|}\right)\log\left(\frac{|12||03|}{|13||02|}\right)-Li_{2}\left(\frac{|12||03|}{|13||02|}\right)+Li_{2}(1)
\]
\[
=:\mathscr{L}\{0123\}.
\]
Writing $t_{i}:=-\frac{b_{i}}{a_{i}}$, note that $\frac{|12||03|}{|13||02|}=\frac{(t_{0}-t_{3})(t_{1}-t_{2})}{(t_{0}-t_{2})(t_{1}-t_{3})}=:CR(t_{0},t_{1},t_{2},t_{3}).$

Now consider a general $\x\cong\PP^{1}$ in $\PP^{4}$ given by\begin{equation}\label{eqn general X P^1}\begin{array}{c}z\\1\end{array}\left(\begin{array}{ccccc}A_0 & A_1 & A_2 & A_3 & A_4 \\ B_0 & B_1 & B_2 & B_3 & B_4 \end{array}\right),
\end{equation}with projections to the coordinate $\PP^{3}$'s (obtained simply by
deleting a column) of the form \eqref{eqn general P^1}. To apply
Theorem \ref{thm Rec Law B}, we will also have to evaluate the correction
terms, or find some way to eliminate them. Again writing $|ij|$ for
the minors, $\{y_{j}\}=\y_{j}=\x\cdot(X_{j})$, and recalling that
on $\PP^{3}$ $\Xi^{0}\left(X_{0}:X_{1}:X_{2}:X_{3}\right)=\pi i\log\left(-\frac{X_{3}}{X_{2}}\right)\delta_{\Gamma_{012}},$
we find that 
\[
\sum_{j=0}^{4}(-1)^{j}\Xi^{0}(y_{j})=\sum_{j=0}^{4}(-1)^{j}\Xi^{0}\left(|j0|:\cdots:\widehat{|jj|}:\cdots:|j4|\right)
\]
\[
=:\mathscr{K}\{01234\}\in\CC
\]
 is anti-invariant under the permutation $\sigma:=(04)(13)$ ``flipping''
\eqref{eqn general X P^1}.

On the other hand, noting that $(03)(12)$ fixes $z:=\frac{|12||03|}{|13||02|},$
and $\frac{|23||10|}{|13||20|}=1-z,$ we have
\[
\tilde{\mathscr{L}}\{0123\}:=\frac{1}{2}\left(\mathscr{L}\{0123\}+\mathscr{L}\{3210\}\right)
\]
\[
=Li_{2}(1)-Li_{2}(z)-\frac{1}{2}\log(1-z)\log(z)
\]
\[
=:L_{2}(z)
\]
which is a version of the \emph{Rogers dilogarithm}. Adding $\frac{1}{2}$
of
\[
0=\frac{1}{2\pi i}\int_{\x}S_{alt}^{2}+\int_{\x}I_{*}^{4}\Xi^{0}
\]
\[
=\sum_{j=0}^{4}(-1)^{j}\mathscr{L}\{0\cdots\widehat{j}\cdots4\}+\mathscr{K}\{01234\}
\]
to $\sigma_{*}$ of itself therefore gives, with $z_{j}:=z(y_{j})=-\frac{B_{j}}{A_{j}},$\begin{equation} \label{eq!!27}
0=\sum_{j=0}^{4}(-1)^{j}L_{2}\left(CR(z_{0},\ldots,\widehat{z_{j}},\ldots,z_{4})\right)
\end{equation}which is the other classic form of the 5-term relation.
\begin{rem}
\label{remark !} The $\left\{ f_{j}(\x):=CR_{\x}(z_{0},\ldots,\widehat{z_{j}},\ldots,z_{4})\right\} $
define 5 rational functions on $Gr(2,5)$. Pulling them back to a
suitable open $U\subset\CC^{2}$ via\begin{align*}
g:U &\to Gr(2,5) \\
(x,y) &\mapsto span\left\{ \begin{matrix}(y^{-1},1,-1,0,1)\\ (1,x,-1,1,0) \end{matrix} \right\}
\end{align*}produces the functions $\left\{ F_{j}:=f_{j}\circ g\right\} =$
\[
x\,,\; y\,,\;\tfrac{y}{x}\,,\;\tfrac{y(1-x)}{x(1-y)}\,,\;\tfrac{1-x}{1-y}\,,
\]
whose level sets yield the Bol 5-web $\mathcal{B}_{5}$.%
\footnote{See \cite{He} for basic material on webs.%
} Clearly \eqref{eq!!27} pulls back to the variant\begin{equation} \label{p27add1}
0 = \sum_{j=0}^4 (-1)^j L_2\left( F_j (x,y)\right)
\end{equation}of \eqref{ASrel}, which is the most interesting of the 6 independent
abelian relations of $\mathcal{B}_{5}$. Moreover, the terms of \eqref{p27add1}
are described by\begin{equation} \label{p27add2}
2\pi i L_2\left( F_j (x,y)\right) = \int_{[g(x,y)]} S_3^{\Delta}\left(\cdots \widehat{X_j} \cdots \right) + \int_{\sigma [g(x,y)]} S^{\Delta}_3 \left( \cdots \widehat{X_j}\cdots \right) .
\end{equation}
\end{rem}

\section{A functional equation for $Li_{3}$}

Turning to the trilogarithm
\[
Li_{3}(z)=\int_{0}Li_{2}(z)\frac{dz}{z},
\]
we will show that the Kummer-Spence relation%
\footnote{This form of the relation is obtained from \cite[p. 177]{Le} by substituting
$u=\frac{ab-b+1}{ab+1}$, $v=\frac{1}{ab+1}$; it is the complex-valued
version of \cite[(1.17)]{Go3}.%
} \begin{multline}\label{6eq1}
-Li_3\left(\tfrac{ab-b+1}{ab^2}\right) -Li_3\left(\tfrac{ab-b+1}{a}\right)-Li_3\left(a(ab-b+1)\right)
\\
+2\left\{ Li_3(a)+Li_3(b)+Li_3(-ab)+Li_3(ab-b+1)-Li_3(1)\right.
\\
\left. +Li_3\left(\tfrac{ab-b+1}{-b}\right) + Li_3\left( \tfrac{ab-b+1}{ab}\right) \right\}  =
\log^2(a)\log(-ab)-\tfrac{\pi^2}{3}\log(a)-\tfrac{1}{3}\log^3(a)
\end{multline}essentially follows from Reciprocity Law B. The ``essentially''
means that we will work modulo \emph{degenerate} terms (i.e. products
of $\log$ and $Li_{2}$ in rational-function arguments) and assume
the relations\begin{equation}\label{6eq2}
Li_3(y)=Li_3(\tfrac{1}{y})+2\zeta(2)\log(y)-\tfrac{1}{6}\log^3(y)-\tfrac{i\pi}{2}\log^2(y)
\end{equation}\begin{multline}\label{6eq3}
Li_3(x)+Li_3(1-x)+Li_3(\tfrac{x}{x-1})= \\
Li_3(1)+Li_2(1)\log(1-x)-\tfrac{1}{2}\log(-x)\log^2(1-x)+\tfrac{1}{6}\log^3(1-x)
\end{multline}from \cite[pp. 154-5]{Le}. We shall denote $Li_{3}(z)=:[z]$, so
that \eqref{6eq2} and \eqref{6eq3} become $[y]\equiv[y^{-1}]$ and\begin{equation}\label{6eq4}
[x]+[1-x]+[1-\tfrac{1}{x}]\equiv [1]
\end{equation}modulo degenerates.

In contrast to the situation (of a $\PP^{1}$ in $\PP^{4}$) worked
out in $\S5$, the direct application of Reciprocity Law B to 
\[
\x:=\,\text{a completely general }\PP^{2}\text{ in }\PP^{6}
\]
seems somewhat intractable. Working modulo degenerates allows us to
eliminate the $\int_{\PP^{1}}\Xi^{2}$ integrals, which (by Lemma
\ref{lemma 2 on Xi}) take the same form as the $S_{3}^{\Delta}$
integrals worked out in $\S5$. At this point we can relax the notion
of general position in Definition \ref{defn general pos} to \emph{proper
intersections for $k\geq2$}:
\begin{lem}
\label{lem 7a}Let $\mathcal{U}\subset Gr(3,7)$ be the analytic open
on which $\x$ and its projections to the coordinate hyperplanes are
general in this sense. (This is the complement of a real codimension-1
subset.) Then writing $S_{5,\hat{j}}^{\Delta}:=S_{5}^{\Delta}\left(X_{0}:\cdots:\widehat{X_{j}}:\cdots X_{6}\right)$,
the integrals $\int_{\x}S_{5,\hat{j}}^{\Delta}$ are (complex) analytic
as a function of $\x\in\mathcal{U}$.
\end{lem}
With this relaxed notion, the projectivized row space $\x_{a,b,c}(\cong\PP^{2})$
of \begin{equation}\label{6eq5}
\begin{array}{c} 1\\ x\\ y \end{array}\left(\begin{array}{ccccccc} 1 & 0 & c & -1 & 0 & 0 & 1\\ a & 1 & 0 & -1 & 0 & 1 & 0\\ 0 & b & 1 & -1 & 1 & 0 & 0 \end{array}\right)
\end{equation}is general in $\PP^{6}$ for sufficiently general $(a,b,c)\in\CC^{3}$.
By Lemma \ref{lem 7a}, the seven integrals\begin{equation}\label{6eq6}
\mathscr{I}_j(a,b,c):=\frac{1}{(2\pi i)^2}\int_{\x_{a,b,c}}S^{\Delta}_{5,\hat{j}}\;,\;\;\; j=0,\ldots,6
\end{equation}are each analytic on the complement $U_{j}\subset\CC^{3}$ of some
real codimension-1 subset. (This is just the locus where the projection
of $\x_{a,b,c}$ to $\PP_{\hat{j}}^{5}$ is general.) Since we do
not know if $\cap U_{j}\subseteq U:=\mathcal{U}\cap\CC^{3}$ is connected,
and we prefer to evaluate the $\mathscr{I}_{j}$ in different regions,
we have to consider the ``jumps'' in the $\mathscr{I}_{j}$ as we
cross over $\CC^{3}\backslash U_{j}$.
\begin{lem}
\label{lem 7b} The jumps in the $\{\mathscr{I}_{j}\}$ (across real
codimension-1 components of $\CC^{3}\backslash U_{j}$) are degenerate.
\end{lem}
Lemmas \ref{lem 7a} and \ref{lem 7b} are proved in the first appendix
to this section.

The upshot of this discussion is that we have \begin{equation}\label{6eq7}
\sum_{i=0}^6 (-1)^j \mathscr{I}_j \equiv 0
\end{equation}modulo degenerates, and that (in \eqref{6eq7}) we may evaluate each
$\mathscr{I}_{j}$ anywhere in $U_{j}$ and analytically continue
the results to a common neighborhood in $U$. To apply Reciprocity
Law B in this form, we shall begin by choosing real subloci $\mathcal{A}_{j}\subset U_{j}\cap\mathbb{R}^{3}$
on which the integrand $\mathscr{I}_{j}$ has only one nonvanishing
term:\begin{flalign*}
\mathcal{A}_0 := & \left\{ a\in\RR ;\, b\in (\tfrac{1}{2},1) ;\, c\in(\tfrac{1}{1-b},\infty) \right\} &
\\
\mathcal{A}_1 := & \left\{ a\in(0,\tfrac{1}{2}) ;\, b\in\RR ;\, c\in(-\infty,1-\tfrac{1}{a}) \right\} &
\\
\mathcal{A}_2 := & \left\{ a\in(0,1) ;\, b\in(1,\tfrac{1}{1-a}) ;\, c\in\RR \right\} &
\\
\mathcal{A}_3 =  \cdots = \mathcal{A}_6 :=& \left\{ a,b,c\in\RR_{<0} ;\, |abc|>1 \right\} .&
\end{flalign*}For example, $\mathscr{I}_{0}$ is the integral (on $\PP^{2}$) of
\[
\tfrac{1}{(2\pi i)^{2}}S_{5}^{\Delta}(\begin{array}[t]{cccccc}
x+by: & c+y: & -(1+x+y): & y: & x: & 1\\
X_{1} & X_{2} & X_{3} & X_{4} & X_{5} & X_{6}
\end{array})\;=
\]
\begin{multline*}
\log\left(\tfrac{y}{1+x+y}\right)\text{dlog}\left(\tfrac{-x}{y}\right)\wedge\text{dlog}\left(\tfrac{-1}{x}\right)\cdot\delta_{T_{-\frac{c+y}{x+by}}\cap T_{\frac{1+x+y}{c+y}}}
\\
+ \{\cdots\}\cdot\delta_{\boxed{T_{-\frac{c+y}{x+by}}\cap T_{\frac{1+x+y}{c+y}}\cap T_{\frac{y}{1+x+y}}}}
\end{multline*}and the boxed intersection is empty on $\mathcal{A}_{0}$. 

More uniformly, writing\begin{flalign*}
\tau_0:=T_{-\frac{c+y}{x+by}}\cap T_{\frac{x+y+1}{c+y}}, 
\\
\tau_1:=T_{-\frac{c+y}{1+ax}}\cap T_{\frac{x+y+1}{c+y}}, 
\\
\tau_2:=T_{-\frac{x+by}{1+ax}}\cap T_{\frac{x+y+1}{x+by}}, 
\\
\tau_3=\cdots =\tau_6:=T_{-\frac{x+by}{1+ax}}\cap T_{-\frac{c+y}{x+by}},
\end{flalign*}we have that\begin{flalign*}
\tau_i \cap T_{\frac{y}{x+y+1}} = \emptyset & \text{ on }\mathcal{A}_i\;\;(i=0,1,2)
\\
\tau_3 \cap T_{\frac{-y}{c+y}} = \emptyset & \text{ on }\mathcal{A}_3
  \\
\tau_i \cap T_{\frac{y}{x+y+1}} = \emptyset & \text{ on }\mathcal{A}_i\;\;(i=4,5,6).
\end{flalign*}Hence the $\mathscr{I}_{j}$ are integrals of $\log(\cdot)\text{dlog}(\cdot)\wedge\text{dlog}(\cdot)$-forms
on the (positively-oriented)
regions:
\[\includegraphics[scale=0.7]{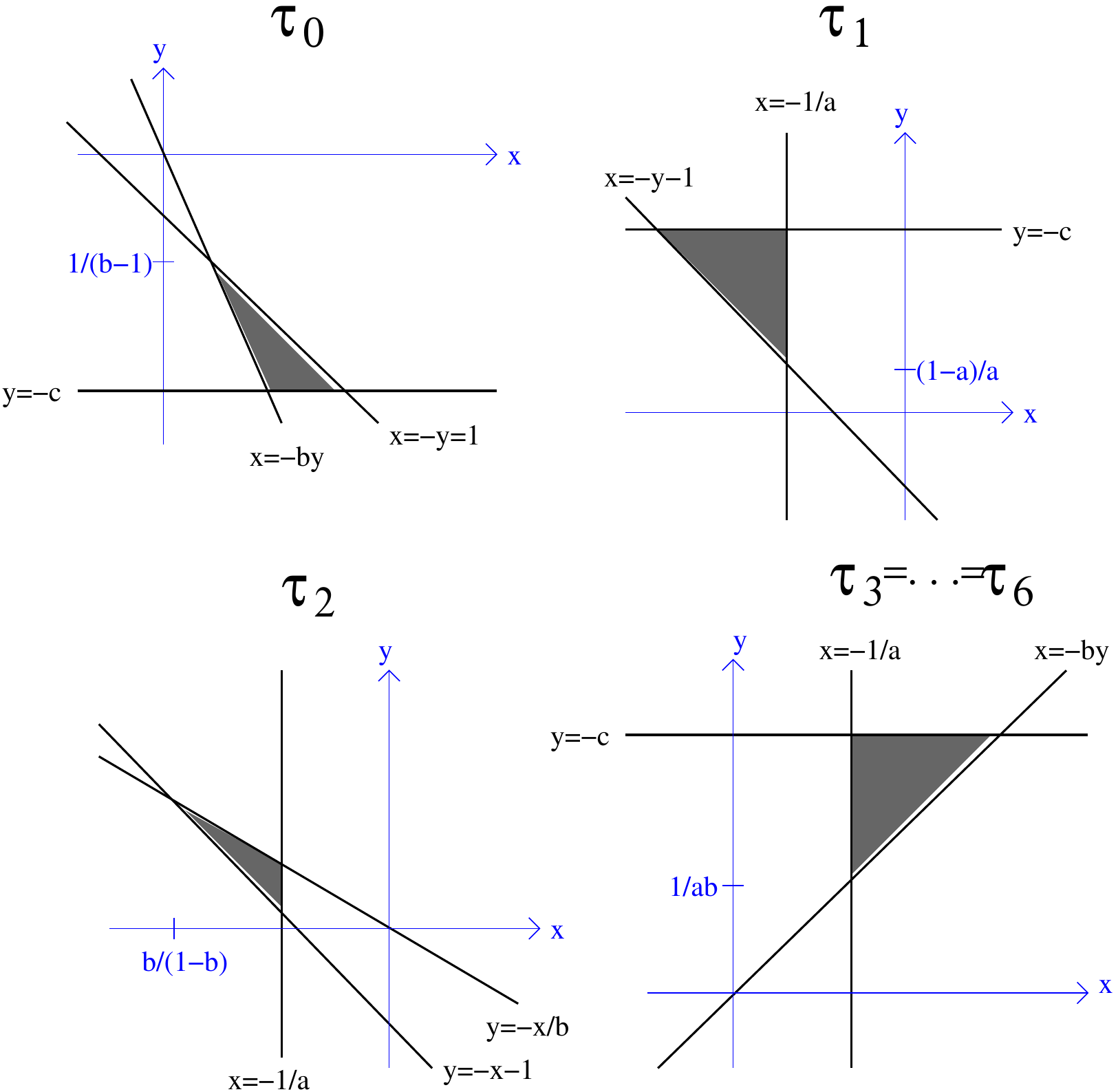}\]
Namely, 
we have \begin{equation}\label{6eq8}
\;\mathscr{I}_0 = -\int_{y=-c}^{\frac{1}{b-1}}\int_{x=-by}^{-y-1}\frac{\log(y)-\log(x+y+1)}{xy} dx\,dy,
\end{equation}\begin{equation}
-\mathscr{I}_1 = \int_{y=\frac{1-a}{a}}^{-c}\int_{x=-y-1}^{-\frac{1}{a}}\frac{\log(y)-\log(x+y+1)}{xy}dx\,dy,
\end{equation}\begin{equation}
\mathscr{I}_2 = -\int_{x=\frac{b}{1-b}}^{-\frac{1}{a}}\int_{y=-x-1}^{-\frac{x}{b}}\frac{\log(y)-\log(x+y+1)}{xy}dx\,dy,
\end{equation} and
\[
-\mathscr{I}_{3}+\mathscr{I}_{4}-\mathscr{I}_{5}+\mathscr{I}_{6}=\int_{\tau_{3}}\begin{array}[t]{c}
\left\{ -\log\left(\tfrac{-y}{c+y}\right)\text{dlog}\left(\tfrac{-x}{y}\right)\wedge\text{dlog}\left(\tfrac{-1}{x}\right)\right.\\
+\log\left(\tfrac{x+y+1}{c+y}\right)\text{dlog}\left(\tfrac{x}{x+y+1}\right)\wedge\text{dlog}\left(\tfrac{-1}{x}\right)\\
-\log\left(\tfrac{x+y+1}{c+y}\right)\text{dlog}\left(\tfrac{y}{x+y+1}\right)\wedge\text{dlog}\left(\tfrac{-1}{y}\right)\\
\left.+\log\left(\tfrac{x+y+1}{c+y}\right)\text{dlog}\left(\tfrac{y}{x+y+1}\right)\wedge\text{dlog}\left(\tfrac{-x}{y}\right)\right\} 
\end{array}
\]
\begin{equation}\label{6eq11}
=\int_{y=\frac{1}{ab}}^{-c}\int_{x=-\frac{1}{a}}^{-by}\frac{\log(y)-\log(x+y+1)}{xy}dx\,dy.
\end{equation}Evaluating these integrals as described in the second appendix below
yields (mod degenerates)
\[
\mathscr{I}_{0}\equiv[b]-[1-b]+2[c]+[1-c]-\left[\tfrac{bc-c+1}{b}\right]+[bc-c+1]+\left[\tfrac{bc-c+1}{bc}\right]-[1]
\]
\[
-\mathscr{I}_{1}\equiv[a]-[1-c]-\left[\tfrac{ac-a+1}{c}\right]+\left[\tfrac{ac-a+1}{ac}\right]+[ac-a+1]
\]
\[
\mathscr{I}_{2}\equiv[a]-\left[1-\tfrac{1}{b}\right]-\left[\tfrac{ab-b+1}{a}\right]+\left[\tfrac{ab-b+1}{ab}\right]+[ab-b+1]
\]
\[
\left.\begin{array}{c}
-\mathscr{I}_{3}+\mathscr{I}_{4}\\
-\mathscr{I}_{5}+\mathscr{I}_{6}
\end{array}\right\} \equiv\begin{array}[t]{c}
2[-ab]-2[c]-\left[\tfrac{ab-b+1}{ab^{2}}\right]+\left[\tfrac{ab-b+1}{ab}\right]+2\left[\tfrac{ab-b+1}{-b}\right]\\
+\left[\tfrac{bc-c+1}{b}\right]-[bc-c+1]-\left[\tfrac{bc-c+1}{bc}\right]-[a(ab-b+1)]\\
+[ab-b+1]+\left[\tfrac{ac-a+1}{c}\right]-\left[\tfrac{ac-a+1}{ac}\right]-[ac-a+1].
\end{array}
\]
Adding these and making use of \eqref{6eq4}, all the terms involving
$c$ cancel and we have 
\[
0\equiv\begin{array}[t]{c}
2[a]+2[b]+2[-ab]-2[1]+2\left[\tfrac{ab-b+1}{ab}\right]+2\left[\tfrac{ab-b+1}{-b}\right]\\
+2[ab-b+1]-\left[\tfrac{ab-b+1}{a}\right]-\left[\tfrac{ab-b+1}{ab^{2}}\right]-[a(ab-b+1)],
\end{array}
\]
which recovers the $Li_{3}$ terms in \eqref{6eq1}.
\begin{rem}
(a) If we take $(a,b,c)$ equal in \eqref{6eq8}-\eqref{6eq11}, they
are just the integrals of $\log\left(\tfrac{y}{x+y+1}\right)\tfrac{dx}{x}\wedge\tfrac{dy}{y}$
over a sum of four \emph{canceling} triangles. This gives a quicker
proof of \eqref{6eq1}, but of course there is something to the fact
that Reciprocity Law B produces the right combination of triangles.

(b) For fixed $c$, \eqref{6eq5} gives a map $G$ from $U\subset\CC^{2}$
to $Gr(3,7)$ analogous to $g$ in Remark \ref{remark !}. In contrast
to the Bol 5-web situation, there is clearly no nice relationship
between the leaves of the Kummer-Spence 9-web \cite{Pi} and the $\int_{[G(a,b)]}S_{5,\hat{j}}^{\Delta}$
integrals (of which there are only $7$). One could still ask whether
the functions $a$, $b$, $-ab$, $\tfrac{ab-b+1}{ab}$, $\tfrac{ab-b+1}{-b}$,
$ab-b+1$, $\tfrac{ab-b+1}{a}$, $\tfrac{ab-b+1}{ab^{2}}$, $a(ab-b+1)$
are $G$-pullbacks of some natural functions on $Gr(3,7)$, perhaps
related to the higher cross-ratios (of $6$ points on $\PP^{2}$)
of Goncharov \cite{Go4}.

(c) The cancellation of all terms involving $c$ was a surprise to
the authors. We do expect that some variant of \eqref{6eq5} should
lead to a similar proof of Goncharov's 22-term relation \cite{Go3},
but leave this as a problem for others.
\end{rem}

\subsection*{Appendix I to $\S6$: Proof of Lemmas \ref{lem 7a} and \ref{lem 7b}}

Write $\mathcal{U}_{j}\subset Gr(3,7)$ for the region on which the
projection of $\x$ to $\PP_{\hat{j}}^{5}$ is general in the weaker
sense. Note that the $\tilde{\mathscr{I}}_{j}(\x):=\int_{\x}S_{5,\hat{j}}^{\Delta}$
are equisingular hence continuous for $\x\in\mathcal{U}_{j}$.%
\footnote{We don\textquoteright{}t need to worry about the properness of intersection
$\x\cap T_{-\frac{X_{1}}{X_{0}}}$ because the terms $\int_{\x\cap T}\log\cdot\text{dlog}\wedge\text{dlog }\wedge\text{dlog}$
vanish by Hodge type.%
} We will show that the restriction of $\tilde{\mathscr{I}}_{j}$ to
$P\cap\mathcal{U}_{j}$ is holomorphic, for $P\subset Gr(3,7)$ an
arbitrary $\PP^{1}$ (with coordinate $t$). Let $\mu\subset P\cap\mathcal{U}_{j}$
be a small disk with boundary $\del\mu=:\gamma$. By Morera's theorem,
it will suffice to check that $\oint_{\x\in\gamma}\tilde{\mathscr{I}}_{j}(\x)dt=0$.

Pulling $S_{5,\hat{j}}^{\Delta}$ back to the total space $\cup_{\x\in\mathcal{U}_{j}}\x=:\tilde{\mathcal{X}}_{j}\overset{\tilde{\pi}_{j}}{\to}\mathcal{U}_{j}$,
we compute using Proposition \ref{prop telescope}
\[
\int_{\pi^{-1}(\gamma)}\tfrac{S_{5,\hat{j}}^{\Delta}}{(2\pi i)^{2}}\wedge dt=\int_{\pi^{-1}(\mu)}d\left[\tfrac{S_{5,\hat{j}}^{\Delta}}{(2\pi i)^{2}}\right]\wedge dt
\]
\begin{equation} \label{appIIeq}
= \sum_{j'\neq j}(\pm 1)\int_{\pi^{-1}(\mu )} (\rho_j ')_* \tfrac{S^{\Delta}_{4,\widehat{jj'}}}{2\pi i} \wedge dt .
\end{equation}(Here we also use the fact that equisingularity $\implies$ $T_{5,\hat{j}}^{\Delta}\cap\pi^{-1}(\mu)=\emptyset$.)
The terms of \eqref{appIIeq} take the form
\[
\tfrac{1}{2\pi i}\int_{\mu\times\PP^{1}}R_{4}\left(-\tfrac{f_{1}}{f_{0}},-\tfrac{f_{2}}{f_{1}},-\tfrac{f_{3}}{f_{2}},-\tfrac{f_{4}}{f_{3}}\right)\wedge dt,
\]
where the $f_{i}$ are linear forms algebraic in $t\in\mu$ (and the
$\PP^{1}$ corresponds to $\x\cap\PP_{\hat{j}}^{5}$). By Hodge type,
the $\int_{\mu\times\PP^{1}}\log\cdot\text{dlog}\wedge\text{dlog}\wedge\text{dlog}\wedge dt$
and $\int_{\mu\times\PP^{1}\cap T}\log\cdot\text{dlog}\wedge\text{dlog}\wedge dt$
terms vanish; while the $\mu\times\PP^{1}\cap T\cap T\cap T$ vanish
by equisingularity. In fact, the $\mu\times\PP^{1}\cap T_{-\frac{f_{1}}{f_{0}}}\cap T_{-\frac{f_{2}}{f_{1}}}$
vanish also, by equisingularity and linearity of the $f_{i}$ (so
that $T_{-\frac{f_{1}}{f_{0}}}\cap\x\cap\PP_{\hat{j}}^{5}$ and $T_{-\frac{f_{2}}{f_{1}}}\cap\x\cap\PP_{\hat{j}}^{5}$
are open segments in $\PP^{1}\cong\x\cap\PP_{\hat{j}}^{5}$ meeting
only at an endpoint). Hence \eqref{appIIeq} is zero and Lemma \ref{lem 7a}
is proved.

The proof of Lemma \ref{lem 7b} is similar. Pulling $S_{5,\hat{j}}^{\Delta}$
back to the total space $\cup\x_{a,b,c}:=\mathcal{X}\overset{\pi}{\to}\CC^{3}\subset Gr(3,7)$,
we have for $p,q\in U_{j}$
\[
\int_{\pi^{-1}(p)}\tfrac{S_{5,\hat{j}}^{\Delta}}{(2\pi i)^{2}}-\int_{\pi^{-1}(q)}\tfrac{S_{5,\hat{j}}^{\Delta}}{(2\pi i)^{2}}=\int_{\pi^{-1}(\vec{qp})}d\left[\tfrac{S_{5,\hat{j}}^{\Delta}}{(2\pi i)^{2}}\right]
\]
\[
\underset{\QQ(3)}{\equiv}\sum_{j'\neq j}(\pm1)\int_{\pi^{-1}(\vec{qp})}(\rho_{j'})_{*}\tfrac{S_{4,\widehat{jj'}}^{\Delta}}{2\pi i}
\]
by Proposition \ref{prop telescope}. The only possible contributions
to a jump arise when a $\delta_{T}$-term in $S_{4}^{\Delta}$ lies
over a component of $\CC^{3}\backslash U_{j}$ crossed by $\vec{qp}$,
and then the contribution is a combination of $(2\pi i)\int_{\PP^{1}}\log\cdot\text{dlog}\cdot\delta_{T}$
and $(2\pi i)^{2}\int_{\PP^{1}}\log\cdot\delta_{T}\cdot\delta_{T}$
integrals, which are obviously degenerate.

\subsection*{Appendix II to $\S6$: Evaluating \eqref{6eq8}-\eqref{6eq11}}

In the course of the computation, we must frequently evaluate integrals
of the form \begin{equation}\label{6eq12}
\int\frac{\log(a-x)\log(b-x)}{x}dx,
\end{equation}\begin{equation}\label{6eq13}
\int\Li_2\left(\frac{1}{a(1+x)}\right)\frac{dx}{x}.
\end{equation}Begin by rewriting \eqref{6eq12} as \begin{equation}\label{6eq14}
\int\tfrac{\log^{2}(a-x)+\log^{2}(b-x)}{2x}dx-\int\tfrac{\log^{2}\left(\tfrac{a-x}{b-x}\right)}{2x}dx.
\end{equation}The first integral may be done by parts twice (e.g. $u=\log^{2}(a-x)$
and $dv=\tfrac{dx}{x}$; then substitute $t=a-x$ and take $u=\log t$,
$dv=\log(a-t)\tfrac{dt}{t}=(\log a)\tfrac{dt}{t}-d(\Li_{2}(\tfrac{t}{a}))$),
which yields\begin{multline*}
-\Li_{3}\left(1-\tfrac{x}{a}\right)+\Li_{2}\left(1-\tfrac{x}{a}\right)\log(a-x)+\tfrac{1}{2}\log(\tfrac{x}{a})\log^{2}(a-x)
\\
-\Li_{3}\left(1-\tfrac{x}{b}\right)+\Li_{2}\left(1-\tfrac{x}{b}\right)\log(b-x)+\tfrac{1}{2}\log(\tfrac{x}{b})\log^{2}(b-x)
\end{multline*}For the second integral in \eqref{6eq14}, substituting $y=\frac{a-x}{b-x}$
gives 
\[
-\tfrac{1}{2}(a-b)\int\tfrac{\log^{2}(y)}{(1-y)(a-yb)}dy,
\]
whereupon repeated integration by parts (starting with $u=\log^{2}y$,
$dv=dy/((1-y)(a-yb))$) yields\begin{multline*}
-\Li_{3}(y)+\Li_{3}(\tfrac{by}{a})-\log(y)\Li_{2}(\tfrac{by}{a})
\\
-\tfrac{1}{2}\log^{2}(y)\log(1-\tfrac{yb}{a})+\Li_{2}(y)\log(y)+\tfrac{1}{2}\log(1-y)\log^{2}(y).
\end{multline*}The $\Li_{3}$ terms from \eqref{6eq12} are therefore
\[
\left[\tfrac{b(a-x)}{a(b-x)}\right]-\left[\tfrac{a-x}{b-x}\right]-\left[1-\tfrac{x}{a}\right]-\left[1-\tfrac{x}{b}\right].
\]

For \eqref{6eq13}, taking $u=\Li_{2}\left(\tfrac{1}{a(1+x)}\right)$
and $dv=\tfrac{dx}{x}$ gives 
\[
\log(x)\Li_{2}\left(\tfrac{1}{1+x}\right)-\int\tfrac{\log\left(\tfrac{1}{x+1}\right)\log(x)}{x+1}
\]
whereupon substituting $t=x+1$ puts the last integral in the form
\eqref{6eq12}. This yields \eqref{6eq13}$\equiv$
\[
-\left[\tfrac{ax+a-1}{x}\right]+\left[\tfrac{ax+a-1}{ax}\right]+\left[1-a(x+1)\right]+[-x]+[x+1].
\]

So for example, $\mathscr{I}_{0}$ breaks into
\[
-\int_{y=-c}^{\frac{1}{b-1}}\int_{x=-by}^{-y-1}\tfrac{\log(y)}{xy}dx\, dy\equiv\left[c\right]-\left[\tfrac{1}{1-b}\right],
\]
which is straightforward, and
\[
\int_{y=-c}^{\frac{1}{b-1}}\int_{x=-by}^{-y-1}\tfrac{\log(x+y+1)}{xy}dx\, dy=
\]
\[
\int_{y=-c}^{\frac{1}{b-1}}\tfrac{-\zeta(2)+\log(-1-y)\log(1+y)-\log(-by)\log(1+y)+\Li_{2}\left(\frac{by}{1+y}\right)}{y}dy.
\]
A substitution brings the second and fourth terms of the last integral
into the forms \eqref{6eq12} and \eqref{6eq13} respectively, and
the other two terms are easy.

\section*{Appendix: On a modification of Goncharov's regulator by Jos� Ignacio Burgos-Gil }

As we have seen in Remark \ref{rem Go}, the  map denoted as
$\cP^{\bullet}(n)$ (Goncharov regulator) in \cite{Go2} fails to be a
morphism of complexes,
hence it does not define a regulator map. By contrast, the cubical
version of the same map is a morphism of complexes and does define a
regulator map.

In the paper \cite{BFT} it is proved that the cubical version of
Goncharov  regulator is compatible with Beilinson regulator.
In the same paper it is also stated that the simplicial version of
Goncharov regulator is compatible with Beilinson regulator. The proof in
\emph{loc. cit.} is based on the assumption that Goncharov map is a
morphism of complexes. Since this is not the case, \cite[Theorem
7.12]{BFT} is not true. The aim of this appendix is to show that,
following the ideas of the present paper, one can define a variant of
the simplicial version of Goncharov regulator that is actually a
morphism of complexes and 
that the resulting regulator is compatible with Beilinson's one.

\medskip
Before we continue, a consumer warning. Being this text  a special  inclusion,
most of the arguments are either  sketched or 
just  pointers to the literature where similar arguments are used,
except for the crucial point that the map in question  is a 
morphism of complexes - that  being the main point of this appendix.

\medskip

As in the previous sections we
restrict ourselves to the case of smooth projective schemes.  
We start by fixing  notation. We will use projective coordinates
$(X_{0}:\dots :X_{n})$ in $\PP^{n}$ and projective coordinates
$((Y_{1}:Z_{1}),\dots ,(Y_{n}:Z_{n}))$ of $(\PP^{1})^{n}$. Recall that
$H_{n}\subset \PP^{n}$ denotes the hyperplane of equation
$\sum_{i}X_{i}=0$. We write $\Delta ^{n}=\PP^{n}\setminus H_{n}$ and
let $\Delta $ denote the cosimplicial scheme $(\Delta ^{n})_{n\ge 0}$ with
the usual faces and degeneracies. We denote
$\square^{1}=\PP^{1}\setminus \{(1:1)\}$ and $\square^{n}=(\square
^{1})^{n}$. These schemes form a cocubical scheme $\square=(\square
^{n})_{n\ge 0}$. 

Recall from from \cite{BKK} that to each Dolbeault complex $A^{\ast}$ we can
associate a Deligne complex $\mathscr{D}_{A}^{\ast}(\ast)$ and if $A$
is a Dolbeault algebra, then $\mathscr{D}_{A}^{\ast}(\ast)$ has an
algebra structure that is associative up to homotopy and
commutative. The main examples are:
\begin{itemize}
\item $A=E^{\ast}(\x)$, the complex of smooth complex valued
  differential forms on a smooth complex variety $\x$. This is a
  Dolbeault algebra and the corresponding Deligne algebra is denoted
  $\mathscr{D}^{\ast}(\x,\ast)$.
\item $A=D^{\ast}(\x)$, the complex of currents on $\x$. This is a
  Dolbeault complex and the corresponding Deligne complex is
  denoted $\mathscr{D}_{D}^{\ast}(\x,\ast)$. 
\end{itemize}

If $\x$ is equidimensional of dimension $d$,
the current associated to every differential form gives a
quasi-isomorphism of Deligne complexes
\begin{equation}
  \label{eq:14}
\mathscr{D}^*(\x,p)\xrightarrow{[\cdot]}
\mathscr{D}_{D}^*(\x,p),\qquad \alpha \mapsto [\alpha],
\end{equation}
where $[\alpha ]$ is the current
\begin{displaymath}
  [\alpha ](\omega )=\frac{1}{(2\pi i)^{d}} \int_{X}\omega \land \alpha.
\end{displaymath}

In \cite{Bu} it is proved that, when $\x$ is projective, the complexes
$\mathscr{D}^{\ast}(\x,\ast)$ and  $\mathscr{D}_{D}^{\ast}(\x,\ast)$
compute Deligne cohomology of $\x$. Moreover, there are explicit
formulas for the  differential $d_{\mathscr{D}}$, for the product $\bullet$ on the
Deligne algebra and for homotopy equivalences between these complexes and a
real variant of the complex $C_{\mathcal{D}}^{\bullet}(\x;\QQ(p))$ of
Section \ref{sec:abel-jacobi-maps}.

If $u_{1},\dots,u_{m}$ are smooth functions on an open subset of a
complex variety, following \cite{Wa}, for $i=1,\dots,m$, we
write
\begin{multline} \label{eq:12}
  \cS_{m}^{i}(u_{1},\dots,u_{m})=\\(-2)^{m}\sum_{\sigma \in \mathfrak S_m}
(-1)^{|\sigma|}u_{\sigma (1)}\partial u_{\sigma (2)}\land\dots\land
\partial u_{\sigma (i)}\land \bar \partial u_{\sigma
  (i+1)}\land\dots\land \bar \partial u_{\sigma({m})},
\end{multline}
and
\begin{equation} \label{eq:13}
  T_{m}(u_{1},\dots,u_{m})=
  \frac{1}{2m!}\sum_{i=1}^{m}(-1)^{i}\cS_{m}^{i}(u_{1},\dots u_{m}).
\end{equation}
We also write $T_{0}=1$.

If $f_1,\dots,f_n\in \CC(\x)^{\times}$ are rational functions, we write
\begin{displaymath}
  \cT_{m}(f_{1},\dots,f_{m})=T_{m}\Big(\frac{-1}{2}\log f_{1}\overline
  f_{1},\dots ,\frac{-1}{2}\log f_{n}\overline f_{n}\Big) 
\end{displaymath}
Since $\cT_m$ is a multilinear alternate function,
$\cT_m(f_{1},\dots,f_{m})$ only depends on the class
\begin{displaymath}
  f_1\land\dots\land f_m\in \Lambda ^{m} \CC(\x)^{\times}.
\end{displaymath}
Hence we use also the notation
\begin{displaymath}
  \cT_m(f_{1}\land\dots\land f_{m})=\cT_m(f_{1},\dots,f_{m}).
\end{displaymath}

Thanks to the fact that $\cT_{m}$ is multilinear and alternate, the
differential form $\cT_m(f_{1},\dots,f_{m})$ is
always locally integrable as a singular form on $\x$. In fact, in
order for a non-locally  integrable term to appear in the expansion of
$\cT_m$ we need that the divisors of the functions $f_i$ have common
components. But this case can always be avoided using multilinearity
and antisymmetry. We denote by $[\cT_m(f_{1},\dots,f_{m})]$ the
associated current on $\x$. 

Recall that, if $Z\subset \x$ is a codimension one point of $X$, the
valuation of $Z$ induces a map $\Res_{Z}\colon \Lambda ^{m}
\CC(\x)^{\times} \to \Lambda ^{m-1} \CC(Z)^{\times}$. For instance, if
\begin{displaymath}
  \ord_{Z}(f_{2})=\dots = \ord_{Z}(f_{m})=0,
\end{displaymath}
then
\begin{displaymath}
  \Res_{Z}(f_{1}\land\dots\land
  f_{m})=\ord_{Z}(f_{1})(f_{2}\land\dots\land f_{m})|_{Z}. 
\end{displaymath}
This property and the fact that it is multilinear and alternate
determine $\Res_{Z}$.  

We now define the current $([\cT_{m-1}]\circ \Res)(f_1\land\dots\land
f_m)$ on $\x$ by 
\begin{displaymath}
([\cT_{m-1}]\circ \Res )(f_1\land\dots\land f_m)=
\sum_{Z\in \x^{(1)}}(\iota_{Z})_{\ast}[
\cT_{m-1}\left(\Res_Z(f_1\land\dots\land f_m)\right)],  
\end{displaymath}
where $\x^{(1)}$
is the set of irreducible closed
subvarieties of codimension one on $\x$ and
$\iota_{Z}\colon \widetilde Z\longrightarrow \x$ is the composition of a
resolution of singularities of $Z$ with the natural map to $X$.

The Deligne differential of $\cT_m(f_1\land\dots\land f_m)$ is described
as follows (\cite[Proposition 2.8]{Go2}, and \cite[Proposition
5.16]{BFT}): 

\begin{prop}\label{prop:dif}
Let $f_{1},\dots,f_{m}\in \mathbb{C}(\x)^{\times}$. Then,
as differential forms on $\x$ with logarithmic singularities
\begin{equation}
  \label{eq:6}
  d_{\mathscr{D}}\cT_{m}(f_{1}\land\dots\land f_{m})=0.
\end{equation}
As currents on $X$ 
\begin{equation}\label{eq:8}
  d_{\mathscr{D}}[\cT_{m}(f_{1}\land\dots\land f_{m})]=
  -([\cT_{m-1}]\circ \Res)(f_{1}\land\dots\land f_{m}).
\end{equation}
\end{prop}

Of particular interest for us are the forms
\begin{gather*}
  r_{n}=\cT_{n}\Big(\frac{X_{1}}{X_{0}},\dots,\frac{X_{n}}{X_{0}}\Big),\\
  r'_{n}=\cT_{n}\Big(\frac{(X_1+\cdots +X_n)}{-X_0} , \frac{(X_2+\cdots
    +X_n)}{-X_1},\ldots,\frac{X_n}{-X_{n-1}}\Big), \\
  w_n=\cT_{n}\Big(\frac{Z_{1}}{Y_{1}},\dots,\frac{Z_{n}}{Y_{n}}\Big) 
\end{gather*}
The form $r_n$ is, up to a normalization factor the form considered by
Goncharov in \cite{Go2}, the form $r_{n}'$ is a modification of $r_n$
that satisfies the additional property
\begin{equation}
  \label{eq:2}
  r_n'|_{H_{n}}=0.
\end{equation}
This is the main difference between $r_{n}$ and $r_{n}'$. The forms
$w_n$ where considered by Wang in \cite{Wa} and have been used to
construct the cubical version of Goncharov regulator. 
The forms $r_n$ and $r_n'$ are locally integrable forms on $\PP^{n}$
and $w_n$ is locally integrable in $(\PP^{1})^n$. We
denote the associated currents as $[r_n]$, $[r'_n]$ and $[w_n]$. 

We denote by $\overline \rho _{i}\colon \PP^{n-1}\to \PP^{n}
$, $n\ge 0$, $0\le i\le n$ the maps given by
\begin{displaymath}
  \overline \rho _{i}(X_{0}:\dots:X_{n-1})=
  (X_{0}:\dots:X_{i-1}:0:X_{i}:\dots,X_{n-1}).
\end{displaymath}
There is a commutative diagram
\begin{displaymath}
  \xymatrix{\Delta ^{n-1}\ar@{^{(}->}[r] 
    \ar[d]_{\rho  _{i}}&\PP^{n-1} \ar[d]_{\overline \rho  _{i}}\\
    \Delta ^{n}\ar@{^{(}->}[r] &\PP^{n}}
\end{displaymath}
where $\rho _{i}$ are the faces in the cosimplicial structure of
$\Delta $. Analogously we write $\overline \delta_{j}^{i}\colon
(\PP^{1})^{n-1}\to (\PP^{1})^{n} $, $n\ge 1$, $1\le i\le n$, $j=0,1$
the maps given by  
\begin{align*}
  \overline \delta^{i}_{0}&((Y_{1}:Z_{1}),\dots,(Y_{n-1}:Z_{n-1}))\\ &=
  ((Y_{1}:Z_{1}),\dots,(Y_{i-1}:Z_{i-1}),(1:0),(Y_{i}:Z_{i}),\dots,(Y_{n-1}:Z_{n-1})),\\
  \overline \delta^{i}_{1}&((Y_{1}:Z_{1}),\dots,(Y_{n-1}:Z_{n-1}))\\ &=
  ((Y_{1}:Z_{1}),\dots,(Y_{i-1}:Z_{i-1}),(0:1),(Y_{i}:Z_{i}),\dots,(Y_{n-1}:Z_{n-1})).
\end{align*}
Again, there is a commutative diagram 
\begin{displaymath}
  \xymatrix{\square ^{n-1}\ar@{^{(}->}[r] 
    \ar[d]_{\delta ^{i}_{j}}&(\PP^{1})^{n-1} \ar[d]_{\overline \delta ^{i}_{j}}\\
    \square ^{n}\ar@{^{(}->}[r] &(\PP^{1})^{n}}
\end{displaymath}
where $\delta ^{i}_{j}$ are the faces in the cocubical structure of
$\square$. 

Using Proposition \ref{prop:dif}, we can compute the
Deligne differential of the previous forms and currents.

\begin{prop}\label{prop:1} The following formulas hold.
  \begin{equation}
    \label{eq:3}
    d_{\mathscr{D}} r_n=d_{\mathscr{D}} r'_n=d_{\mathscr{D}} w_n=0
  \end{equation}
  \begin{align}
    \label{eq:4}
     d_{\mathscr{D}} [r_n]&=\sum_{i=0}^{n}(-1)^{i}(\overline \rho
                            _{i})_{\ast}[r_{m-1}],\\ 
     d_{\mathscr{D}} [r'_n]&=\sum_{i=0}^{n}(-1)^{i}(\overline \rho
                            _{i})_{\ast}[r'_{m-1}],\label{eq:5}\\
     d_{\mathscr{D}} [w_n]&=\sum_{i=0}^{n}(-1)^{i}((\overline \delta
                            ^{i}_{0})_{\ast}[w_{m-1}]-(\overline \delta
                            ^{i}_{1})_{\ast}[w_{m-1}]).\label{eq:7}  
  \end{align}
\end{prop}
\begin{proof}
  We prove only equation \eqref{eq:5}. For $i=0,\dots,n$, let $D_{i}$,
  be the divisor of equation $X_{i}=0$ and for $i=1,\dots,n-1$, let
  $E_{i}$ be the divisor of equation $X_{i}+\dots+X_{n}=0$. For $i<n$,
  We have
  \begin{multline*}
    \Res_{D_{i}}\Big(\frac{(X_1+\cdots +X_n)}{-X_0}\land \ldots
    \land\frac{X_n}{-X_{n-1}}\Big)\\= (-1)^{i+1}
    \Big(\frac{(X_1+\cdots +X_n)}{-X_0}\land\dots\land
    \widehat{\frac{(X_{i+1}+\cdots
    +X_n)}{-X_i}}\land\ldots\land \frac{X_n}{-X_{n-1}}\Big)\Big|_{D_{i}},
  \end{multline*}
  where the symbol $\widehat {\phantom{m}} $ means that the term is
  omitted. For $i=n$,
  \begin{multline*}
    \Res_{D_{n}}\Big(\frac{(X_1+\cdots +X_n)}{-X_0}\land \ldots
    \land\frac{X_n}{-X_{n-1}}\Big)\\=(-1)^{n-1} 
    \Big(\frac{(X_1+\cdots +X_{n-1})}{-X_0}\land\dots\land
    \frac{X_{n-1}}{-X_{n-2}}\Big)\Big|_{D_{n}} 
  \end{multline*}
  Moreover, for $i=1,\dots,n-1$,
  \begin{multline*}
    \Res_{E_{i}}\Big(\frac{(X_1+\cdots +X_n)}{-X_0}\land \ldots
    \land\frac{X_n}{-X_{n-1}}\Big)\\= (-1)^{i}
    \Big(\frac{(X_1+\cdots +X_n)}{-X_0}\land\dots\land
    \widehat{\frac{(X_{i}+\cdots 
        +X_n)}{-X_{i-1}}}\land\ldots\land
    \frac{X_n}{-X_{n-1}}\Big)\Big|_{E_{i}}\\
    =0
   \end{multline*}
because the restriction of the function
$(X_{i+1}+\dots+X_{n})/(-X_{i})$ to $E_{i}$ is one. Note that this is
the key point that makes this form work.

Finally, for a divisor $Z\in (\PP^{n})^{(1)}$ different form the
previous ones, we have
\begin{displaymath}
  \Res_{Z}\Big(\frac{(X_1+\cdots +X_n)}{-X_0}\land \ldots
    \land\frac{X_n}{-X_{n-1}}\Big)=0.
\end{displaymath}
From these formulas and Proposition \ref{prop:dif} the equation
\eqref{eq:5} follows.
\end{proof}

Let $\x$ be a smooth projective variety over a field $k\subset
\CC$. We denote by $Z^{p}_{\Delta}(\x,n)_{0}$ and $Z^{p}_{\square}(\x,n)_{0}$
the normalized simplicial and cubical Bloch's higher Chow
complexes. In the simplicial case this is the group denoted
$N_{\Delta}^p(\x, n)$ before Proposition \ref{ML2} and in the cubical
case is defined for instance in \cite[Section 4.2]{BFT}.
 By \cite{Lv} 
we know that both complexes have the same homology groups: the higher
Chow groups $CH^{p}(\x,n)$. We recall
the argument in \cite{Lv} to prove that the simplicial  and cubical
versions of Bloch's higher Chow groups agree. 

Let $Z^{p}_{\square,\Delta}(\x,n,m)$ denote the group of
codimension $p$
cycles in $\x\times \square ^{n}\times \Delta ^{m}$ that meet properly
all the proper faces of $\x\times \square ^{n}\times \Delta
^{m}$. Then $Z^{p}_{\square,\Delta}(\x,\cdot,\cdot)$ is a
cocubical-cosimplicial abelian group. We
denote by $Z^{p}_{\square,\Delta}(\x,\ast,\ast)_{0}$ the associated
normalized double complex and by $Z^{p}_{\square,\Delta}(\x,\ast)_{0}$
the corresponding simple complex.
Then
\cite[Theorem 4.7]{Lv} states that both natural inclusions
\begin{displaymath}
  Z^{p}_{\Delta}(\x,\ast)_{0}\to
  Z^{p}_{\square,\Delta}(\x,\ast)_{0}\quad\text{and}\quad  
  Z^{p}_{\square}(\x,\ast)_{0}\to Z^{p}_{\square,\Delta}(\x,\ast)_{0}
\end{displaymath}
are quasi-isomorphisms. 

Since the family of currents $([r_m'])_{m\ge 0}$  is singular along the
hyperplanes $E_{i}$ introduced in the proof of Proposition
\ref{prop:1}, in order to use it to define a
regulator map, we
need to restrict the class of cycles we use. We denote by $L$ the
hyperplane arrangement
\begin{displaymath}
  L=E_{1}\cup\dots\cup E_{n-1}
\end{displaymath}
and by $Z^{p}_{\Delta ,L}(X,n)$ the group of
codimension $p$ cycles of $X\times \Delta ^{n}$ that intersect
properly all the finite intersections among the divisors $X\times D_{i}$,
$i=0,\dots,n$ and $X\times E_{j}$, $j=1,\dots,n-1$. These groups form a
simplicial complex and we write $Z^{p}_{\Delta
  ,L}(X,\ast)_0 $ for the normalized complex. It is a subcomplex of
$Z^{p}_{\Delta }(X,\ast)_{0}$. 

Using the same argument as in the proof of \cite[Lemma 8.14]{KL} (see
also the proof of \cite[Theorem 2.7]{Lv}) one can prove
\begin{lem}\label{lemm:1}
  The inclusion of complexes $Z^{p}_{\Delta
  ,L}(X,\ast)_{0}\hookrightarrow Z^{p}_{\Delta }(X,\ast)_{0}$ is a
quasi-isomorphism.
\end{lem}

We now denote by $Z^{p}_{\square,\Delta,L}(\x,n,m)$ the group of
codimension $p$
cycles in $\x\times \square ^{n}\times \Delta ^{m}$ that meet properly
all the finite intersections among the faces and the divisors $X\times
\square^{n}\times E_{j}$. Again it is a cocubical-cosimplicial abelian
group and we
denote by $Z^{p}_{\square,\Delta,L}(\x,\ast,\ast)_{0}$ the associated
normalized double complex and by $Z^{p}_{\square,\Delta,L}(\x,\ast)_{0}$
the corresponding simple complex. Combining the arguments of
\cite[Theorem 4.7]{Lv} and of with \cite[Lemma 8.14]{KL} one can prove
that the natural inclusions
\begin{displaymath}
  i_{\Delta }\colon Z^{p}_{\Delta,L}(\x,\ast)_{0}\to
  Z^{p}_{\square,\Delta,L}(\x,\ast)_{0}\quad\text{and}\quad i_{\square}\colon
  Z^{p}_{\square}(\x,\ast)_{0}\to Z^{p}_{\square,\Delta,L}(\x,\ast)_{0}
\end{displaymath}
are quasi-isomorphisms. 

As in \cite{Go2}, the forms $r_n$, $w_n$ and $r'_n$ determine, for
each $p,n\ge 0$, maps 
\begin{align*}
  &\cP_{\Delta }\colon Z^{p}_{\Delta }(X,n)_{0}\to
  \mathscr{D}_{D}^{2p-n}(\x,p),\\ 
  &\cP_{\square }\colon Z^{p}_{\square }(X,n)_{0}\to
  \mathscr{D}_{D}^{2p-n}(\x,p),\\ 
  &\cP'_{\Delta}\colon Z^{p}_{\Delta,L }(X,n)_{0}\to
  \mathscr{D}_{D}^{2p-n}(\x,p).
\end{align*}
We just give the definition of $\cP'_{\Delta}$ the others already having been
defined in the literature. For instance $\cP_{\Delta }$ is, up
to a normalization factor, the map denoted $\cP^{\bullet}(n)$ in
\cite{Go2}.  Let $Y$ be an irreducible subvariety 
of $\x\times \PP ^{n}$ 
that is 
not contained in any of the divisors  $\x\times D_{i}$, $i=0,\dots,n,$
or $\x\times E_{j}$, $j=1,\dots,n-1$, and
$\widetilde Y$ a resolution of singularities of $Y$. Put 
$\iota\colon \widetilde Y\to X$ and $q\colon \widetilde Y\to \PP^{n}$ for
the maps induced by the two projections. Then $q^{\ast}r_{n}'$ is a
differential form, on an open subset of $\widetilde Y$, that is locally
integrable on $\widetilde Y$. Since the map $\iota$ is proper, we obtain a
well defined current
\begin{displaymath}
  \cP'_{\Delta }(Y)=\iota_{\ast}[q^{\ast}r_{n}']\in \mathscr{D}_{D}^{2p-n}(\x,p).
\end{displaymath}

If $Y$ is an irreducible subvariety of $\x\times \Delta
  ^{n}$  of codimension $p$,
  such that the cycle $[Y]$ belons to  $Z^{p}_{\Delta,L }(X,n)$, we
  write $\overline Y$ for the closure of $Y$ in $\x\times \PP ^{n}$
  and define
  \begin{displaymath}
    \cP'_{\Delta }([Y])=\cP'_{\Delta }(\overline Y).
  \end{displaymath}
The map $\cP'_{\Delta }$ is extended  by linearity to the whole 
$Z^{p}_{\Delta,L}(X,n)$ and then restricted to
$Z^{p}_{\Delta,L}(X,n)_{0}$ to give the desired map. 

We have seen in Remark \ref{rem Go} that, for a fixed $p$, the maps
$\cP_{\Delta }$ do not form a morphism of complexes. Nevertheless we have
\begin{thm}\label{thm:1}
  For every integer $p\ge 0$, the maps
  \begin{displaymath}
    \cP_{\square }\colon Z^{p}_{\square }(X,\ast)_{0}\to
  \mathscr{D}_{D}^{2p-\ast}(\x,p),\ \text{and}\  
  \cP'_{\Delta}\colon Z^{p}_{\Delta,L }(X,\ast)_{0}\to
  \mathscr{D}_{D}^{2p-\ast}(\x,p)
  \end{displaymath}
are morphisms of complexes.
\end{thm}
\begin{proof}
  We give the proof for $\cP'_{\Delta}$ being the other one
  analogous. Let $Y$ be an irreducible subvariety of $\x\times \Delta
  ^{n}$  of codimension $p$,
  such that the cycle $[Y]$ belongs to  $Z^{p}_{\Delta,L }(X,n)$. In
  particular $Y$ is not contained in any of the divisors $\x\times D_{i}$
  or $\x\times E_{j}$. Let $\partial$ denote the differential in the complex  
  $Z^{p}_{\Delta,L }(X,\ast)$ and write
  \begin{displaymath}
    \overline \partial  [\overline Y]= \sum
    _{i=0}^{n}(-1)^{i}\overline{\rho} _{i}^{\ast}[\overline Y].
  \end{displaymath}
  This is a cycle on $\x\times \PP^{n-1}$. Note that $\partial [Y]$ is
  the restriction of $\overline \partial  [\overline Y]$ to $\x\times
      \Delta ^{n-1}$, that is,
   \begin{displaymath}
    \partial [Y] =\overline \partial  [\overline Y]\big|_{\x\times
      \Delta ^{n-1}}. 
  \end{displaymath}
  We decompose $ \overline \partial  [\overline Y]=A+B$, where the
  cycle $A$ gathers all the components of $ \overline \partial
  [\overline Y]$ not contained in $H_{n-1}$ and $B$ gathers the
  components contained in $H_{n-1}$. Thus
  \begin{displaymath}
    A=\overline {(\partial [Y])}
  \end{displaymath}
  is the closure of $\partial [Y]$ in $\x\times \PP^{n-1}$. 
  Proposition \ref{prop:1} implies readily that
  \begin{displaymath}
    d_{\mathscr{D}}\cP'_{\Delta }([Y])=\cP'_{\Delta }(\overline \partial
    [\overline Y]) = \cP'_{\Delta }(A) + \cP'_{\Delta }(B).
  \end{displaymath}
  The key point, that fails for the family of currents $(r_{n})_{n\ge 0}$ is that, since
  $r'_{n-1}$ vanishes on $H_{n-1}$, then $\cP'_{\Delta }(B)=0$. Thus
  \begin{displaymath}
    d_{\mathscr{D}}\cP'_{\Delta }([Y])=\cP'_{\Delta }(A)= \cP'_{\Delta }(\partial
    [Y])
  \end{displaymath}
  proving the result for $\cP_{\Delta }'$. The proof for
  $\cP_{\square}$ is similar.
\end{proof}

Thanks to Theorem \ref{thm:1} and Proposition \ref{prop:1}, we obtain
morphisms
\begin{displaymath}
  \cP_{\square},\cP'_{\Delta }\colon CH^{p}(X,n)\longrightarrow
 H_{\cD}^{2p-n}(\x,\RR(p)).
\end{displaymath}

By \cite[Theorem 7.8]{BFT} we know that the map $\cP_{\square}$ is
compatible with Beilinson regulator. To prove that $\cP'_{\Delta }$
is also compatible with Beilinson regulator we compare $\cP_{\square}$
and $\cP'_{\Delta }$. To this end we introduce the locally integrable
forms
\begin{displaymath}
  M_{n,m}=\cT_{n+m}\Big(\frac{Z_{1}}{Y_{1}},\dots,\frac{Z_{n}}{Y_{n}},
  \frac{(X_1+\cdots 
    +X_n)}{-X_0} ,\ldots,\frac{X_n}{-X_{n-1}}\Big) 
\end{displaymath}
on $(\PP^{1})^n\times \PP^{n}$, that are a hybrid between the forms
$w_n$ and $r'_m$. Following the same steps used to define $\cP'_{\Delta }$
and to prove Theorem \ref{thm:1}, we obtain
\begin{prop}
  The family of forms $(M_{n,m})_{n,m}$ induce a morphism of complexes
  \begin{displaymath}
    \cP'_{\square,\Delta}\colon Z^{p}_{\square,\Delta,L}(\x,\ast)_{0} \to
  \mathscr{D}_{D}^{2p-\ast}(\x,p).
  \end{displaymath}
\end{prop}

Now we can state and prove the main result of this appendix.

\begin{thm}\label{thm:2} The maps
  \begin{displaymath}
    \cP_{\square},\cP'_{\Delta }\colon CH^{p}(X,n)\longrightarrow
    H_{\cD}^{2p-n}(\x,\RR(p))    
  \end{displaymath}
  agree. In consequence, the map $\cP'_{\Delta }$ is compatible with
  Beilinson regulator.
\end{thm}
\begin{proof}
  A direct computation shows that the diagram
  \begin{displaymath}
    \xymatrix{ Z^{p}_{\square,\Delta,L}(\x,\ast)_{0} 
      \ar[rd]^{\cP'_{\square,\Delta}}&
      Z^{p}_{\square}(\x,\ast)_{0} \ar[d]^{\cP_{\square}}
      \ar[l]_{i_{\square}}\\  
      Z^{p}_{\Delta,L}(\x,\ast)_{0} \ar[r]_{\cP'_{\Delta}}
      \ar[u]^{i_{\Delta }}&  
      \mathscr{D}_{D}^{2p-\ast}(\x,p)
    }
  \end{displaymath}
  is commutative. Then the theorem follows from the fact that
  $i_{\square}$ and $i_{\Delta }$ are quasi-isomorphisms.
\end{proof}

{\begin{rem} 
We remark that the simplicial regulator $\cP'_{\Delta}$
just described also agrees with the composite
\begin{equation}\label{JJJ}
CH^p(\x,n)_{\QQ} \overset{AJ^{p,n}_{\Delta,X}}{\longrightarrow} H^{2p-n}_{\mathcal{H}}\left( \x^{an}_{\CC},\QQ(p)\right) \overset{\pi_{\mathbb{R}}}{\longrightarrow}
H^{2p-n}_{\mathcal{H}}\left( \x^{an}_{\CC},\RR(p)\right),
\end{equation}
where (as $\x$ is smooth projective)
\[
H^{2p-n}_{\mathcal{H}}\left( \x^{an}_{\CC},\RR(p)\right) \simeq \frac{H^{2p-n}\left( \x^{an}_{\CC},\CC\right)}
{F^pH^{2p-n}\left( \x^{an}_{\CC},\CC\right) + H^{2p-n}\left( \x^{an}_{\CC},\RR(p)\right)}
\]
\[
\xrightarrow{\pi_{p-1 \ \simeq}}
\frac{H^{2p-n}\left( \x^{an}_{\CC},\RR(p-1)\right)}{\pi_{p-1}
F^pH^{2p-n}\left( \x^{an}_{\CC},\CC\right)},
\]
with $\pi_{p-1}$ induced by the canonical projection
\[
 \CC = \RR(p)\oplus \RR(p-1) \to \RR(p-1).
 \]
Indeed, any class in $CH^p(X,n)$ has a representative $\mathfrak{Z}\in\ker(\partial)\subset Z^p_{\Delta}(X,n)_0\cap Z^p_{\Delta,\mathbb{R}}(X,n)$, and we have $\cP_{\square}(\mathfrak{Z}^{\square}) = \cP'_{\Delta}(\mathfrak{Z})$ (by Theorem \ref{thm:2}) and $AJ^{p,n}_X(\mathfrak{Z}^{\square}) = AJ_{\Delta,X}^{p,n}(\mathfrak{Z})$ (by the proof of Theorem \ref{AJ Thm}).  So the assertion reduces to the coincidence of $\pi_{\mathbb{R}}\circ AJ^{p,n}_X$ wih the cubical Goncharov regulator $\cP_{\square}$, which was verified in $\S$3.1.1 of \cite{Ke1}.
\end{rem}}

\curraddr{\noun{${}$}\\
\noun{Department of Mathematics, Campus Box 1146}\\
\noun{Washington University in St. Louis}\\
\noun{St. Louis, MO} \noun{63130, USA}}

\email{\emph{${}$}\\
\emph{e-mail}: matkerr@math.wustl.edu}

\curraddr{${}$\\
\noun{Department of Mathematical and Statistical Sciences}\\
\noun{University of Alberta} \\
\noun{Edmonton, Alberta}  \noun{T6G 2G1, CANADA}}

\email{${}$\\
\emph{e-mail}: lewisjd@math.ualberta.ca}

\curraddr{\noun{${}$}\\
\noun{Department of Mathematics}\\
\noun{Harvard University}\\
\noun{Cambridge, MA}  \noun{02138, USA}}

\email{\emph{${}$}\\
\emph{e-mail}: patricklopatto@gmail.com}

\curraddr{\noun{${}$}\\
\noun{Instituto de Ciencias Matematicas (CSIC-UAM-UCM-UCM3)}\\
\noun{Calle Nicol\'as Ca\-bre\-ra~15, Campus UAB, Cantoblanco,}\\
\noun{28049 Madrid 
SPAIN}

\email{${}$\\
\emph{e-mail}: burgos@icmat.es}
\end{document}